\documentclass{amsart}
\usepackage[utf8]{inputenc}
\usepackage{mathrsfs} 

\usepackage{amssymb}
\usepackage{bm}
\usepackage{graphicx}
\usepackage[centertags]{amsmath}
\usepackage{amsfonts}
\usepackage{amsthm}
\usepackage{enumerate}
\usepackage{tocvsec2}
\usepackage{xcolor}
\usepackage[margin=1in]{geometry}

\newtheorem{theorem}{Theorem}[section]
\newtheorem*{theorem*}{Theorem}

\newtheorem{corollary}[theorem]{Corollary}
\newtheorem{lemma}[theorem]{Lemma}
\newtheorem{rem}[theorem]{Remark}

\newtheorem{proposition}[theorem]{Proposition}

\newtheorem{fact}[theorem]{Fact}

\theoremstyle{definition}
\newtheorem{definition}[theorem]{Definition}

\newcommand{\ee}{\varepsilon}
\newcommand{\nn}{\mathbb{N}}
\newcommand{\rr}{\mathbb{R}}

\begin{document}
\title{$\xi$-completely continuous operators and $\xi$-Schur Banach spaces}

\author{R.M. Causey}
\address{Department of Mathematics, Miami University, Oxford, OH 45056, USA}
\email{causeyrm@miamioh.edu}

\author{K.V. Navoyan}
\address{Department of Mathematics, University of Mississippi, Oxford, MS 38655, USA}
\email{KNavoyan@go.olemiss.edu}


\maketitle

\begin{abstract} For each ordinal $0\leqslant \xi\leqslant \omega_1$, we introduce the notion of a $\xi$-completely continuous operator and prove that for each ordinal $0< \xi< \omega_1$, the class $\mathfrak{V}_\xi$ of $\xi$-completely continuous operators is a closed, injective operator ideal which is not surjective, symmetric, or idempotent.  We prove that for distinct $0\leqslant \xi, \zeta\leqslant \omega_1$, the classes of $\xi$-completely continuous operators and $\zeta$-completely continuous operators are distinct.    We also introduce an ordinal rank $\textsf{v}$ for operators such that $\textsf{v}(A)=\omega_1$ if and only if $A$ is completely continuous, and otherwise $\textsf{v}(A)$ is the minimum countable ordinal such that $A$ fails to be $\xi$-completely continuous.  We show that there exists an operator $A$ such that $\textsf{v}(A)=\xi$ if and only if $1\leqslant \xi\leqslant \omega_1$, and there exists a Banach space $X$ such that $\textsf{v}(I_X)=\xi$ if and only if there exists an ordinal $\gamma\leqslant \omega_1$ such that $\xi=\omega^\gamma$.   Finally, prove that for every $0<\xi<\omega_1$, the class $\{A\in \mathcal{L}: \textsf{v}(A) \geqslant \xi\}$ is $\Pi_1^1$-complete in $\mathcal{L}$, the coding of all operators between separable Banach spaces.  This is in contrast to the class $\mathfrak{V}\cap \mathcal{L}$, which is $\Pi_2^1$-complete in $\mathcal{L}$.

\end{abstract}

\section{Introduction}

The stratification of classes of operators using ordinal indices has been a very useful strategy in questions of universality, factorization, and descriptive set theoretic complexity (see \cite{ADST}, \cite{Brooker}, \cite{BC}, \cite{BC2}, \cite{BCFW}, \cite{BF}, \cite{Cweakly}). Such stratification has been undertaken for the classes of strictly singular \cite{ADST}, Asplund \cite{Brooker}, Rosenthal and  unconditionally converging operators \cite{BCFW}, and weakly compact \cite{BF}, \cite{Cweakly} operators.    In this work, we define for each $\xi\leqslant \omega_1$ the notion of a $\xi$-completely continuous operator and study the classes $\mathfrak{V}_\xi$, $0\leqslant \xi\leqslant \omega_1$, where $\mathfrak{V}_\xi$ denotes the class of operators which are $\xi$-completely continuous.  For each $0<\xi<\omega_1$, we also recall the definition of the $\xi$-Banach-Saks operators and study the relationship between the classes $\mathfrak{W}_\xi$ of $\xi$-Banach-Saks operators and $\mathfrak{V}_\xi$.  We recall that $\mathfrak{W}_0$ is defined to be the class of compact operators and $\mathfrak{W}_{\omega_1}$ denotes the class of weakly compact operators.

Our first theorem extends many of the classical results about completely continuous operators.

\begin{theorem} For every $0\leqslant \xi\leqslant \omega_1$, $\mathfrak{V}_\xi$ is a closed, injective operator ideal such that $\mathfrak{V}_\xi=\mathfrak{K}\circ \mathfrak{W}_\xi^{-1}$ and $\mathfrak{V}^\text{\emph{dual}}=(\mathfrak{W}^\text{\emph{dual}}_\xi)^{-1}\circ \mathfrak{K}$.       Furthermore, if $0<\xi$, $\mathfrak{V}_\xi$ fails to be surjective, symmetric, or idempotent.

\end{theorem}

Given an operator $A:X\to Y$, we will show that $A$ is completely continuous if and only if it is $\xi$-completely continuous for every countable $\xi$.  We then let $\textsf{v}(A)=\omega_1$ if $A$ is completely continuous, and otherwise $\textsf{v}(A)$ denote the minimum countable ordinal $\xi$ such that $A$ fails to be $\xi$-completely continuous. Given a Banach space $X$, we let $\textsf{v}(X)=\textsf{v}(I_X)$.   We prove the following regarding the distinctness of the classes $\mathfrak{V}_\xi$.

\begin{theorem} \begin{enumerate}[(i)]\item For an ordinal $\xi$, there exists an operator $A:X\to Y$ such that $\textsf{\emph{v}}(A)=\xi$ if and only if $1\leqslant \xi\leqslant \omega_1$. \item For an ordinal $\xi$, there exists a Banach space $X$ with $\textsf{\emph{v}}(X)=\xi$ if and only if there exists an ordinal $\gamma\leqslant \omega_1$ such that $\xi=\omega^\gamma$.   \end{enumerate}

\label{big2}

\end{theorem}

For the proof of Theorem \ref{big2}, we discuss the notion of a probability block, generalizing the important repeated averages hierarchy introduced in \cite{AMT} and further studied in \cite{AG}. All unfamiliar terminology will be defined later.     For a  countable ordinal $\xi$, we say the probability block $(\mathfrak{P}, \mathcal{P})$, where $\mathfrak{P}$ is a collection of probability measures on $\nn$ and $\mathcal{P}$ is a collection of finite subsets of $\nn$ containing the supports of the members of $\mathfrak{P}$, is  $\xi$-\emph{sufficient} provided that if $\mathcal{G}$ is any regular family with Cantor-Bendixson index not exceeding $\omega^\xi$ and if $\delta>0$, there exists an infinite subset $M$ of $\nn$ such that for all further, infinite subsets $N$ of $M$, the measures $\mathbb{P}\in\mathfrak{P}$ which are supported on $M$ satisfy $\mathbb{P}(E)<\delta$.    We say $(\mathfrak{P}, \mathcal{P})$ is $\xi$-\emph{regulatory} provided that if $f:\nn\times MAX(\mathcal{P})\to \rr$ is a bounded function such which has large averages for all $\mathbb{P}$ coming out of a ``full'' subset of $\mathfrak{P}$, then $f$ must be pointwise large on a set with Cantor-Bendixson index greater than $\omega^\xi$. 

Furthermore, if $(\mathfrak{P}, \mathcal{P})$, $(\mathfrak{Q}, \mathcal{Q})$ are probability blocks, we define a convolution probability block $(\mathfrak{Q}*\mathfrak{P}, \mathcal{Q}[\mathcal{P}])$.    Regarding these notions, we prove the following.

\begin{theorem} If $\xi$ is a countable ordinal and $(\mathfrak{P}, \mathcal{P})$ is a probability block, then if $(\mathfrak{P}, \mathcal{P})$ is $\xi$-sufficient, it is $\xi$-regulatory.

If $\xi, \zeta$ are countable ordinals, $(\mathfrak{P}, \mathcal{P})$ is $\xi$-sufficient, and $(\mathfrak{Q}, \mathcal{Q})$ is $\zeta$-sufficient, then $(\mathfrak{Q}*\mathfrak{P}, \mathcal{Q}[\mathcal{P}])$ is $\xi+\zeta$-sufficient.

\end{theorem}

In the final section of the paper, we recall the codings $\textbf{SB}$ and $\mathcal{L}$ of all separable Banach spaces and all operators between separable Banach spaces, respectively.  These are Polish spaces, and as such, it is often of interest to compute the descriptive set theoretic complexity of given subsets of these classes. Along these lines, we prove the following.

\begin{theorem} For every $0<\xi<\omega_1$, the class $\mathfrak{V}_\xi\cap \mathcal{L}$ is coanalytic complete and therefore non-Borel in $\mathcal{L}$.   Furthermore, $\textsf{\emph{V}}_\xi\cap \textbf{\emph{SB}}$ is coanalytic complete and therefore non-Borel in $\textbf{\emph{SB}}$.    

\end{theorem}

The previous theorem is in contrast to the class $\mathfrak{V}\cap \mathcal{L}$ of completely continuous operators between separable Banach spaces and the class $\textsf{V}\cap \textbf{SB}$ of separable Schur spaces, which are $\Pi_2^1$, and therefore not coanalytic, in the spaces $\mathcal{L}$, $\textbf{SB}$, respectively.

\section{Regular families}

Througout, we let $2^\nn$ denote the power set of $\nn$ and topologize this set with the Cantor topology. Given a subset $M$ of $\nn$, we  let $[M]$ (resp. $[M]^{<\nn}$) denote set of infinite (resp. finite) subsets of $M$.  For convenience, we often write subsets of $\nn$ as sequences, where a set $E$ is identified with the (possibly empty) sequence obtained by listing the members of $E$ in strictly increasing order.  Henceforth, if we assume $(m_i)_{i=1}^r\in [\nn]^{<\nn}$ (resp. $(m_i)_{i=1}^\infty\in [\nn]$), it will be assumed that $m_1<\ldots <m_r$ (resp. $m_1<m_2<\ldots$).  Given $M=(m_n)_{n=1}^\infty\in [\nn]$ and $\mathcal{F}\subset [\nn]^{<\nn}$, we define $$\mathcal{F}(M)=\{(m_n)_{n\in E}: E\in \mathcal{F}\}$$ and $$\mathcal{F}(M^{-1})=\{E: (m_n)_{n\in E}\in \mathcal{F}\}.$$  

Given $(m_i)_{i=1}^r, (n_i)_{i=1}^r\in [\nn]^{<\nn}$, we say $(n_i)_{i=1}^r$ is a \emph{spread} of $(m_i)_{i=1}^r$ if $m_i\leqslant n_i$ for each $1\leqslant i\leqslant r$.   We agree that $\varnothing$ is a spread of $\varnothing$.    We write $E\preceq F$ if either $E=\varnothing$ or $E=(m_i)_{i=1}^r$ and $F=(m_i)_{i=1}^s$ for some $r\leqslant s$.   In this case, we say $E$ is an \emph{initial segment} of $F$.  For $E,F\subset \nn$, we write $E<F$ to mean that either $E=\varnothing$, $F=\varnothing$, or $\max E<\min F$. Given $n\in \nn$ and $E\subset \nn$, we write $n\leqslant E$ (resp. $n<E$) to mean that $n\leqslant \min E$ (resp. $n<\min E$).

We say $\mathcal{G}\subset [\nn]^{<\nn}$ is \begin{enumerate}[(i)]\item \emph{compact} if it is compact in the Cantor topology, \item \emph{hereditary} if $E\subset F\in \mathcal{G}$ implies $E\in \mathcal{G}$,  \item \emph{spreading} if whenever $E\in \mathcal{G}$ and $F$ is a spread of $E$, $F\in \mathcal{G}$,  \item \emph{regular} if it is compact, hereditary, and spreading. \end{enumerate}

Let us also say that $\mathcal{G}$ is \emph{nice} if \begin{enumerate}[(i)]\item $\mathcal{G}$ is regular, \item $(1)\in \mathcal{G}$,  \item for any $\varnothing\neq E\in \mathcal{G}$, either $E\in MAX(\mathcal{G})$ or $E\cup (1+\max E)\in \mathcal{G}$.  \end{enumerate}

If $M\in[\nn]$ and if $\mathcal{F}$ is nice, then there exists a unique, finite, non-empty initial segment of $M$ which lies in $MAX(\mathcal{F})$.  We let $M_\mathcal{F}$ denote this initial segment.    We now define recursively $M_{\mathcal{F},1}=M_\mathcal{F}$ and $M_{\mathcal{F}, n+1}= (M\setminus \cup_{i=1}^n M_{\mathcal{F},i})_{\mathcal{F}}$.    An alternate description of $M_{\mathcal{F},1}, M_{\mathcal{F}, 2}, \ldots$ is that the sequence $M_{\mathcal{F},1}, M_{\mathcal{F},2}, \ldots$ is the unique partition of $M$ into successive sets which are maximal members of $\mathcal{F}$.

If $\mathcal{F}$ is nice and $M\in [\nn]$, then there exists a partition $E_1<E_2<\ldots$ of $\nn$ such that $M_{\mathcal{F},n}=(m_i)_{i\in E_n}$ for all $n\in\nn$.  We define $M^{-1}_{\mathcal{F},n}=E_n$.

Given a topological space $K$ and a subset $L$ of $K$, $L'$ denotes the \emph{Cantor Bendixson derivative} of $L$ consists of those members of $L$ which are not relatively isolated in $L$.    We define by transfinite induction the higher order transfinite derivatives of $L$ by $$L^0=L,$$ $$L^{\xi+1}=(L^\xi)',$$ and if $\xi$ is a limit ordinal, $$L^\xi=\bigcap_{\zeta<\xi}L^\zeta.$$ We recall that $K$ is said to be \emph{scattered} if there exists an ordinal $\xi$ such that $K^\xi=\varnothing$. In this case, we define the \emph{Cantor Bendixson index} of $K$ by $CB(K)=\min\{\xi: K^\xi=\varnothing\}$.    If $K^\xi\neq \varnothing$ for all ordinals $\xi$, we write $CB(K)=\infty$.   We agree to the convention that $\xi<\infty$ for all ordinals $\xi$, and therefore $CB(K)<\infty$ simply means that $CB(K)$ is an ordinal, and $K$ is scattered. 

Of course, if $\xi$ is a limit ordinal, $K$ is a compact  topological space, and $K^\zeta\neq \varnothing$ for all $\zeta<\xi$, then $(K^\zeta)_{\zeta<\xi}$ is a collection of compact subsets of $K$ with the finite intersection property, so $K^\xi=\cap_{\zeta<\xi}K^\zeta\neq \varnothing$.   From this it follows that for a compact topological space, $CB(K)$ cannot be a limit ordinal.

We recall the following, which is well known.  The proof is standard, so we omit it.

\begin{fact} Let $\mathcal{G}\subset [\nn]^{<\nn}$ be hereditary.   The following are equivalent.  \begin{enumerate}[(i)]\item There does not exist $M\in [\nn]$ such that $[M]^{<\nn}\subset \mathcal{G}$.\item $\mathcal{G}$ is compact.  \item $CB(\mathcal{G})<\infty$. \item $CB(\mathcal{G})<\omega_1$. \end{enumerate}

\end{fact}

For each $n\in\nn\cup \{0\}$, we let $\mathcal{A}_n=\{E\in [\nn]^{<\nn}: |E|\leqslant n\}$.    It is clear that $\mathcal{A}_n$ is regular.    Also of importance are the Schreier families, $(\mathcal{S}_\xi)_{\xi<\omega_1}$.  We recall these families.  We let $$\mathcal{S}_0=\mathcal{A}_1,$$ $$\mathcal{S}_{\xi+1}=\{\varnothing\} \cup \Bigl\{\bigcup_{i=1}^n E_i: \varnothing\neq E_i \in \mathcal{S}_\xi, n\leqslant E_1, E_1<\ldots <E_n\Bigr\},$$ and if $\xi<\omega_1$ is a limit ordinal, there exists a sequence $\xi_n\uparrow \xi$ such that $$\mathcal{S}_\xi=\{E\in [\nn]^{<\nn}: \exists n\leqslant E\in \mathcal{S}_{\xi_n+1}\}.$$  We note that the sequence $(\xi_n)_{n=1}^\infty$ has the property that for any $n\in\nn$, $\mathcal{S}_{\xi_n+1}\subset \mathcal{S}_{\xi_{n+1}}$.    The existence of such families with the last indicated property is discussed, for example, in \cite{Concerning}.  

Given two non-empty regular families $\mathcal{F}, \mathcal{G}$, we let $$\mathcal{F}[\mathcal{G}]=\{\varnothing\}\cup \Bigl\{\bigcup_{i=1}^n E_i: \varnothing \neq E_i\in \mathcal{G}, E_1<\ldots <E_n, (\min E_i)_{i=1}^n\in \mathcal{F}\Bigr\}.$$  We let $\mathcal{F}[\mathcal{G}]=\varnothing$ if either $\mathcal{F}=\varnothing$ or $\mathcal{G}=\varnothing$. 

The following facts are collected in \cite{Concerning}.

\begin{proposition} \begin{enumerate}[(i)]\item For any non-empty regular families $\mathcal{F}, \mathcal{G}$, $\mathcal{F}[\mathcal{G}]$ is regular. Furthermore, if $CB(\mathcal{F})=\beta+1$ and $CB(\mathcal{G})=\alpha+1$, then $CB(\mathcal{F}[\mathcal{G}])=\alpha\beta+1$. \item For any $n\in\nn$, $CB(\mathcal{A}_n)=n+1$. \item For any $\xi<\omega_1$, $CB(\mathcal{S}_\xi)=\omega^\xi+1$. \item If $\mathcal{F}$ is regular and $M\in [\nn]$, then $\mathcal{F}(M^{-1})$ is regular and $CB(\mathcal{F})=CB(\mathcal{F}(M^{-1}))$. \item For regular families $\mathcal{F}, \mathcal{G}$, there exists $M\in [\nn]$ such that $\mathcal{F}(M)\subset \mathcal{G}$ if and only if there exists $M\in [\nn]$ such that $\mathcal{F}\subset \mathcal{G}(M^{-1})$ if and only if $CB(\mathcal{F})\leqslant CB(\mathcal{G})$. \end{enumerate} 

\label{deep facts}

\end{proposition}

It is also easy to see that if $\mathcal{F}, \mathcal{G}$ are nice, $\mathcal{F}[\mathcal{G}]$ is nice.

We next recall some results from Ramsey theory.

\begin{theorem}\cite{PR} If $\mathcal{G}\subset [\nn]^{<\nn}$ is regular and if $P_1, \ldots, P_n\subset MAX(\mathcal{G})$ are such that $\cup_{i=1}^n P_i=MAX(\mathcal{G})$, then for any $M\in[\nn]$, there exist $1\leqslant i\leqslant n$ and $N\in [M]$ such that $MAX(\mathcal{G})\cap [N]^{<\nn}\subset P_i$. 

\label{PR}

\end{theorem}

\begin{theorem}\cite{Ga} If $\mathcal{F}, \mathcal{G}\subset [\nn]^{<\nn}$ are hereditary, then for any $M\in [\nn]$, there exists $N\in [M]$ such that either $$\mathcal{F}\cap [N]^{<\nn}\subset \mathcal{G}\text{\ \ \ \ \ or\ \ \ \ \ }\mathcal{G}\cap [N]^{<\nn}\subset \mathcal{F}.$$  

In particular, if $\mathcal{G}$ is regular and $CB(\mathcal{F})<CB(\mathcal{G})$, then for any $M\in [\nn]$, there exists $N\in [M]$ such that $\mathcal{F}\cap [N]^{<\nn}\subset \mathcal{G}$. 
\label{gasp}
\end{theorem}

We now prove an easy consequence of Theorem \ref{gasp} which will be needed later. 

\begin{proposition} If $\mathcal{F}, \mathcal{G}$ are regular families such that $CB(\mathcal{F})\leqslant CB(\mathcal{G})$, then for any $M\in[\nn]$, there exists $N\in[M]$ such that for any $\varnothing \neq E\in \mathcal{F}\cap [N]^{<\nn}$, $E\setminus (\min E)\in \mathcal{G}(M)$.

\label{ironside}
\end{proposition}

\begin{proof} If $\mathcal{F}=\varnothing$, then this is vacuous, so assume $\mathcal{F}\neq \varnothing$. For each $n\in\nn$, let $\mathcal{F}(n)=\{E\in[\nn]^{<\nn}: m<E, (n)\cup E\in \mathcal{F}\}$ and note that $$CB(\mathcal{F}(m_1))<CB(\mathcal{F})\leqslant CB(\mathcal{G}(M))=\min\{CB(\mathcal{G}(M)\cap [N]^{<\nn}): N\in[M]\}.$$   Now let $m_1=\min M$ and choose by Theorem \ref{gasp} some $N_1\in [M]$ such that either $\mathcal{F}(m_1)\cap [N_1]^{<\nn}\subset \mathcal{G}(M)$ or $\mathcal{G}(M)\cap [N_1]^{<\nn}\subset \mathcal{F}(m_1)$.  However, since $CB(\mathcal{F}(m_1))<CB(\mathcal{G}(M)\cap [N_1]^{<\nn})$, the first inclusion must hold.  

Now assuming that $m_1<\ldots <m_n$ and $ N_1\supset \ldots \supset N_n\in[M]$ have been chosen, fix $m_{n+1}\in N_n$ with $m_n<m_{n+1}$.  Arguing as in the previous paragraph, there exists $N_{n+1}\in [N_n]$ such that $\mathcal{F}(m_{n+1})\cap [N_{n+1}]^{<\nn}\subset \mathcal{G}(M)$.    This completes the recursive construction.  

Let $N=(m_n)_{n=1}^\infty$. Then if $\varnothing\neq E\in \mathcal{F}\cap [N]^{<\nn}$, there exists $n\in\nn$ such that $\min E=m_n$, and $E\setminus (m_n)\in \mathcal{F}(m_n)\cap [N_n]^{<\nn}\subset \mathcal{G}(M)$.

\end{proof}

We next recall a special case of the infinite Ramsey theorem, the proof of which was achieved in steps by Nash-Williams \cite{NW}, Galvin and Prikry \cite{GP}, Silver \cite{Silver}, and Ellentuck \cite{Ellentuck}.

\begin{theorem} If $\mathcal{V}\subset [\nn]$ is closed, then for any $M\in [\nn]$, there exists $N\in [M]$ such that either $$[N]\subset \mathcal{V}\text{\ \ \ \ \ or\ \ \ \ \ }[N]\cap \mathcal{V}=\varnothing.$$

\end{theorem}

\section{Quantified weak convergence and weak compactness} 

In what follows, $\mathbb{K}$ denotes the scalar field, either $\mathbb{R}$ or $\mathbb{C}$, and $S_\mathbb{K}=\{t\in \mathbb{K}: |t|=1\}$.   Let $X$ be a Banach space and let $(x_i)_{i=1}^\infty$ be a sequence in $X$. For $\ee>0$, we define three subsets of $[\nn]^{<\nn}$ associated with $(x_i)_{i=1}^\infty$.  We let $$\mathfrak{F}^a_\ee((x_i)_{i=1}^\infty)=\{\varnothing\}\cup \{F\in [\nn]^{<\nn}: (\exists x^*\in B_{X^*})(\forall n\in F)(|x^*(x_n)|\geqslant \ee)\},$$ $$\mathfrak{F}_\ee((x_i)_{i=1}^\infty) = \{\varnothing\}\cup \{F\in [\nn]^{<\nn}: (\exists x^*\in B_{X^*})(\forall n\in F)(\text{Re\ }x^*(x_n)\geqslant \ee)\},$$ $$\mathfrak{F}_\ee^\sigma((x_i)_{i=1}^\infty)= \{\varnothing\}\cup \{F\in [\nn]^{<\nn}:(\forall (\ee_n)_{n\in F}\in S_\mathbb{K}^F)(\exists x^*\in B_{X^*})(\forall n\in F)(\text{Re\ }x^*(\ee_n x_n)\geqslant \ee)\}.$$   

It is evident that $$\mathfrak{F}^\sigma_\ee((x_i)_{i=1}^\infty)\subset \mathfrak{F}_\ee((x_i)_{i=1}^\infty)\subset \mathfrak{F}^a_\ee((x_i)_{i=1}^\infty).$$ It follows from the geometric Hahn-Banach theorem that for $\varnothing\neq F\in [\nn]^{<\nn}$, \begin{enumerate}[(i)]\item $F\in \mathfrak{F}^a_\ee((x_i)_{i=1}^\infty)$ if and only if there exists $(\ee_n)_{n\in F}\in S_\mathbb{K}^F$ such that $$\min \{\|x\|: x\in \text{co}(\ee_n x_n: n\in F)\}\geqslant \ee,$$  \item $F\in \mathfrak{F}_\ee((x_i)_{i=1}^\infty)$ if and only if $$\min \{\|x\|: x\in \text{co}(x_n: n\in F)\}\geqslant\ee,$$ and \item $F\in \mathfrak{F}^\sigma_\ee((x_i)_{i=1}^\infty)$ if and only if for every $(\ee_n)_{n\in F}\in S_\mathbb{K}^F$, $$\min \{\|x\|: x\in \text{co}(\ee_n x_n: n\in F)\}\geqslant \ee.$$  \end{enumerate}

 We next give a partial converse to the inclusions above. For the real case, this result is found in \cite{AMT}.   We  discuss how to make a minor modification of their proof to deduce the complex case. The methods of proof closely follow those in \cite{AMT}.

\begin{theorem} If $(x_i)_{i=1}^\infty$ is a weakly null sequence and $\xi<\omega_1$ is an ordinal, the following are equivalent. \begin{enumerate}[(i)]\item There exist $\ee>0$ and $N\in [\nn]$ such that $\mathcal{S}_\xi(N)\subset \mathfrak{F}^a_\ee((x_i)_{i=1}^\infty)$. \item There exist $\ee>0$ and $N\in [\nn]$ such that $\mathcal{S}_\xi(N)\subset \mathfrak{F}_\ee((x_i)_{i=1}^\infty)$. \item There exist $\ee>0$ and $N\in [\nn]$ such that $\mathcal{S}_\xi(N)\subset \mathfrak{F}_\ee^\sigma((x_i)_{i=1}^\infty)$. \end{enumerate}

\label{convex unc}
\end{theorem}

The first lemma and the succeeding corollary are close modifications of the analogous results from \cite{AG}, so we give only a sketch for completeness.  

\begin{lemma} Suppose $k\in \nn\cup \{0\}$  and $S, T_1, \ldots, T_k\subset \mathbb{C}$ are sets.  If $X$ is a Banach space, $v_1, \ldots, v_k\in X$,  $(u_j)_{j=1}^\infty \subset X$ is a weakly null sequence, and $\delta>0$, then there exists $M\in[\nn]$ such that for any $(n_0, n_1, n_2, \ldots)\in [M]$, if there exist $E,F\subset \{1, \ldots, k\}$, $x^*\in B_{X^*}$, and $s\in \nn\cup \{0\}$ such that $x^*(v_j)\in S$ for all $j\in E$, $x^*(v_j)\in T_j$ for all $j\in F$, and $x^*(u_{n_j})\in S$ for all $1\leqslant j\leqslant s$, then there exists $y^*\in B_{X^*}$ such that $y^*(v_j)\in S$ for all $j\in E$, $y^*(v_j)\in T_j$ for all $j\in F$, $y^*(u_{n_j})\in S$ for all $1\leqslant j\leqslant s$, and $|y^*(u_{n_0})|\leqslant \delta$. 

\label{cu}
\end{lemma}

\begin{proof} For each pair $E,F$ of subsets of $\{1, \ldots, k\}$, let $\mathcal{V}_{E,F}$ denote the set of $M=(m_i)_{i=0}^\infty\in[\nn]$ such that if there exist $x^*\in B_{X^*}$ and $s\in \nn\cup \{0\}$ such that  $x^*(v_i)\in S$ for each $i\in E$,  $x^*(v_i)\in T_i$ for each $i\in F$,  $x^*(u_{m_i})\in S$ for all $1\leqslant i\leqslant s$,  then there exists $y^*\in B_{X^*}$ satisfying these three conditions as well as the condition \item $|y^*(u_{m_0})|\leqslant \delta$.  It is clear that $\mathcal{V}_{E,F}$ is closed, whence for any $L\in[\nn]$, there exists $M\in [L]$ such that either $[M]\cap \mathcal{V}_{E,F}=\varnothing$ or $[M]\subset \mathcal{V}_{E,F}$.   We claim that the former cannot hold. We show this by contradiction, so assume $[M]\cap \mathcal{V}_{E,F}=\varnothing$ and write $M=(m_i)_{i=1}^\infty$.  For each $1\leqslant i\leqslant j$, let $M_{ij}=\{m_i, m_{j+1}, m_{j+2}, \ldots\}$.   Then since $M_{ij}\notin \mathcal{V}_{E,F}$, there exist $x^*_{ij}\in B_{X^*}$ and $s_{ij}\in \nn\cup \{0\}$ such that \begin{enumerate}\item $x^*_{ij}(v_l)\in S$ for each $l\in E$, \item $x^*_{ij}(v_l)\in T_l$ for each $l\in F$, \item $x^*_{ij}(u_{m_{j+l}})\in S$ for each $1\leqslant l\leqslant s_{ij}$,

but there does not exist any functional $x^*\in B_{X^*}$ satisfying (1)-(3) as well as 

\item $|x^*(u_{m_i})|\leqslant \delta$. \end{enumerate}   Now fix $1\leqslant i_j\leqslant j$ such that $s_j:=s_{i_j j}=\max_{1\leqslant i\leqslant j} s_{ij}$ and note that $x^*_j:=x^*_{i_j j}$ satsifies 1-3 for each $1\leqslant i\leqslant j$.    Therefore $x^*_j$ cannot satisfy $4$ for any $i$, which means $|x^*_j(u_{m_i})|>\delta$ for all $1\leqslant i\leqslant j$.   Then if $x^*$ is any weak$^*$-cluster point of $(x^*_j)_{j=1}^\infty$,  $|x^*(u_{m_i})|\geqslant \delta$ for all $i\in\nn$, contradicting the weak nullity of $(u_i)_{i=1}^\infty$.

What the above argument shows is that for each $E,F\subset \{1, \ldots, k\}$ and any $L\in [\nn]$, there exists $M\in [L]$ such that $[M]\subset \mathcal{V}_{E,F}$.   Now if $(E_i, F_i)_{i=1}^t$ is an enumeration of all pairs of subsets of $\{1, \ldots, k\}$, we may recursively select $$\nn\supset L_1\supset \ldots \supset L_t=:M\in [\nn]$$ such that for each $1\leqslant i\leqslant t$, $[L_i]\subset \mathcal{V}_{E_i, F_i}$.    Then this set $M$ is the one we seek.

\end{proof}

In what follows, for $j\in \{0, \ldots, 9\}$ and $\mu>0$, we let $$S_j=\Bigl\{re^{i\theta}: r\geqslant 0, \frac{j\pi}{5}-\frac{\pi}{10}<\theta\leqslant \frac{j \pi}{5}+\frac{\pi}{10}\Bigr\}$$ and $$S_j^\mu= \{z\in S_j: |z|\geqslant \mu\}.$$   

\begin{corollary} Suppose $(x_j)_{j=1}^\infty \subset X$ is a weakly null sequence.    Fix $0<\ee$ and $(\delta_n)_{n=1}^\infty \subset (0, \infty)$.  Then there exists a subsequence $(y_j)_{j=1}^\infty $ of $(x_j)_{j=1}^\infty$ such that for any non-empty set $G\subset \nn$ and $x^*_0\in B_{X^*}$ such that $x^*_0( y_j)\in S^{\frac{\ee}{2}}_0$ for all $j\in G$,  there exists $x^*\in B_{X^*}$ such that $x^*(y_j)\in S^{\frac{\ee}{2}}_0$ for all $j\in G$ and $|x^*(y_j)|\leqslant \delta_j$ for all $j\in \nn\setminus G$.  

\label{bfdi}

\end{corollary}

\begin{proof} Let $S=S^{\ee/2}_0$ and $T_n=\{z\in \mathbb{C}: |z|\leqslant \delta_n\}$.   Let $N_0=\nn$ and $n_0=0$.   Now if $k\in\nn\cup \{0\}$ and if $N_0\supset \ldots \supset N_k\in [\nn]$, $n_0<n_1<\ldots <n_k$ have been chosen such that $n_i\in N_i$ for each $1\leqslant i\leqslant k$, apply Lemma \ref{cu} with $v_i=x_{n_i}$ for each $1\leqslant i\leqslant k$, $\delta=\delta_{k+1}$, and $(u_i)_{i=1}^\infty = (x_i)_{i\in N_k}$ to obtain $N_{k+1}\in [N_k]$, and then fix $n_{k+1}\in N_{k+1}\cap (n_k, \infty)$.    Let $N=(n_i)_{i=1}^\infty$ and let $y_i=x_{n_i}$.

Now suppose that $G\subset \nn$ and $x^*_0\in B_{X^*}$ are such that $x^*_0(x_{n_i})\in S$ for all $i\in G$.    We will select $x^*_1, x^*_2, \ldots$ in $B_{X^*}$ such that for each $j\in \nn\cup \{0\}$, $x^*_j(x_{n_i})\in S$ for all $i\in G$, and $|x^*_j(x_{n_i})|\leqslant \delta_i$ for all $i\in \{1, \ldots, j\}\setminus G$.    We explain how to obtain $x^*_k$ assuming $x^*_{k-1}$ has been chosen.   If $k\in G$, we let $x^*_k=x^*_{k-1}$.   Otherwise we note that there exist $R=(r_i)_{i=0}^\infty \in [\{n_k, n_{k+1}, \ldots\}]\subset [N_k]$ and $s\in \nn\cup \{0\}$ such that $r_0=n_k$ and $G\cap \{k, k+1, \ldots\}=\{r_1, \ldots, r_s\}$.   Here we are using the fact that $G$ is finite, since $(x_{n_i})_{i=1}^\infty$ is weakly null.     Now let $E=\{1, \ldots, k-1\}\cap G$, $F=\{1, \ldots, k-1\}\setminus G$, and note that since $R\in [N_k]$, there exists $x^*_k\in B_{X^*}$ such that $x^*_k(x_{n_i})\in S$ for each $i\in E$, $x^*_k(x_{r_i})\in S$ for each $1\leqslant i\leqslant s$, whence $x^*_k(x_{n_i})\in S$ for all $i\in G$, and $|x^*_k(x_{n_i})|\leqslant \delta_i$ for all $i\in \{1, \ldots, k\}\setminus G$.

Now if $x^*$ is any weak$^*$-cluster point of $(x^*_k)_{k=1}^\infty$, $x^*(x_{n_i})\in S$ for all $i\in G$ and $|x^*(x_{n_i})|\leqslant \delta_i$ for all $i\in \nn\setminus G$.

\end{proof}

\begin{corollary} Fix a weakly null sequence $(x_j)_{j=1}^\infty\subset B_X$ and $\ee>0$. Then there exists a subsequence $(y_j)_{j=1}^\infty$ of $(x_j)_{j=1}^\infty$ such that if  $(\lambda_j)_{j=1}^\infty$ is a scalar sequence such that $\sum_{j=1}^\infty|\lambda_j|\leqslant 1$,  and $\|\sum_{j=1}^\infty \lambda_j y_j\|\geqslant \ee$, then for any $(\ee_j)_{j=1}^\infty\in S_\mathbb{K}^F$, $$\|\sum_{j=1}^\infty \ee_j \lambda_j y_j\|\geqslant \frac{\ee^2}{16000}.$$

\end{corollary}

\begin{proof} We prove only the case $\mathbb{K}=\mathbb{C}$, since the real case is proved (with a better constant) in \cite{AMT}.   Fix a sequence $(\delta_i)_{i=1}^\infty$ such that $\sum_{i=1}^\infty \delta_i<\ee^2/16000$ and let $(y_i)_{i=1}^\infty$ be the subsequence of $(x_i)_{i=1}^\infty$ guaranteed to exist by Corollary \ref{bfdi}.     Assume $(\lambda_j)_{j=1}^\infty$ is a scalar sequence with $\sum_{j=1}^\infty |\lambda_j|\leqslant 1$ such that $\|\sum_{j=1}^\infty \lambda_j y_j\|\geqslant \ee$ and $(\ee_j)_{j=1}^\infty$ is a sequence of unimodular scalars.    For each $j\in \nn$, write $\lambda_j=w_j \sigma_j$, where $w_j\geqslant 0$ and $|\sigma_j|=1$.     Fix $z^*\in B_{X^*}$ such that $$\text{Re\ }z^*\bigl(\sum_{j=1}^\infty \lambda_j y_j\bigr) \geqslant \ee.$$   Let $E=\{j\in \nn: |z^*(y_j)| \geqslant \ee/2\}$ and note that \begin{align*} \ee & \leqslant \text{Re\ }z^*\bigl(\sum_{j=1}^\infty \lambda_j y_j\bigr) \leqslant \sum_{j\in E}w_j + \frac{\ee}{2}\sum_{j=1}^\infty w_j, \end{align*} whence $\ee/2\leqslant \sum_{j\in E}w_j$.    Now for each $j,k,l\in \{0, \ldots, 9\}$, let $$E_{jkl}=\{m\in E: \sigma_m\in S_j, z^*(y_m)\in S^{\ee/2}_k, \ee_m\in S_l\}$$ and note that $$\ee/2 \leqslant \sum_{j,k,l=0}^9 \sum_{m\in E_{jkl}} w_m.$$  From this it follows that there exist $j,k,l\in \{0, \ldots, 9\}$ such that $$\frac{\ee}{2000} \leqslant \sum_{m\in E_{jkl}} w_m.$$   Let $G=E_{jkl}$, $\sigma'=e^{-ij\pi/5}$, $\ee'=e^{-il\pi/5}$, and $x^*=e^{-ik\pi/5}z^*$.   Since $x^*\in B_{X^*}$ and $x^*(y_m)\in S^{\ee/2}_0$ for all $m\in G$, there exists $y^*\in G$ such that $y^*(y_m)\in S^{\ee/2}_0$ for all $m\in G$ and $|y^*(y_m)|\leqslant \delta_m$ for all $m\in \nn\setminus G$.     Let $$Q=\Bigl\{re^{i\theta}: r\geqslant \ee/2, -\frac{\pi}{3}<\theta\leqslant \frac{\pi}{3}\Bigr\}$$ and note that $\text{Re\ }z\geqslant \ee/4$ for all $z\in Q$.   Note also that $\sigma'\sigma_m \ee'\ee_m y^*(y_m)\in Q$ for each $m\in G$.   Therefore \begin{align*} \|\sum_{m=1}^\infty \lambda_m \ee_m y_m\| & = \|\sum_{m=1}^\infty w_m \sigma'\sigma_m \ee'\ee_m y_m\|  \geqslant \text{Re\ }y^*\Bigl(\sum_{m=1}^\infty w_m \sigma'\sigma_m \ee'\ee_m y_m\Bigr) \\ & \geqslant \sum_{m\in G} w_m\text{Re\ }(\sigma'\sigma_m \ee'\ee_m y^*(y_m)) - \sum_{m=1}^\infty \delta_m \\ & \geqslant \frac{\ee}{4}\sum_{m\in G} w_m - \frac{\ee^2}{16000} \geqslant \frac{\ee^2}{8000}-\frac{\ee^2}{16000} = \frac{\ee^2}{16000}. \end{align*}

\end{proof}

Theorem \ref{convex unc} follows immediately from the previous result.

\begin{definition} Fix a countable ordinal $\xi$. For a Banach space $X$ and a sequence $(x_i)_{i=1}^\infty\subset X$, we say $(x_i)_{i=1}^\infty$ is $\xi$-\emph{weakly null} if there do not exist $\ee>0$ and $N\in [\nn]$ such that $\mathcal{S}_\xi(N)\subset \mathfrak{F}^a_\ee((x_i)_{i=1}^\infty)$.  

We say $(x_i)_{i=1}^\infty$ is $\omega_1$-\emph{weakly null} if it is weakly null. 

We say $(x_i)_{i=1}^\infty$ is $\xi$-\emph{weakly convergent to} $x$ if $(x_i-x)_{i=1}^\infty$ is $\xi$-weakly null.   

We say $(x_i)_{i=1}^\infty$ is $\xi$-\emph{weakly convergent} if it is $\xi$-weakly convergent to $x$ for some $x\in X$.

\end{definition}

We include a brief description of basic facts concerning these notions.

\begin{proposition} \begin{enumerate}[(i)]\item Any sequence is norm null if and only if it is $0$-weakly null.  \item Any sequence which is $\xi$-weakly null for some $\xi<\omega_1$ is weakly null.  \item Any weakly null sequence is $\xi$-weakly null for some $\xi<\omega_1$.  \item Any sequence which is $\xi$-weakly null for some $\xi<\omega_1$ is $\zeta$-weakly null for every $\xi<\zeta<\omega_1$.  \end{enumerate}

\label{basic facts}
\end{proposition}

\begin{proof}$(i)$ We note that for a sequence $(x_i)_{i=1}^\infty$, $(n)\in \mathcal{F}_\ee((x_i)_{i=1}^\infty)$ if and only if $\|x_n\|\geqslant \ee$. Then $(x_i)_{i=1}^\infty$ is $0$-weakly null if and only if $\{n\in \nn: \|x_i\|\geqslant \ee\}$ is finite for all $\ee>0$, and $(x_i)_{i=1}^\infty$ fails to be $0$-weakly null if and only if $\{n\in\nn: \|x_i\|\geqslant \ee\}$ is infinite for some $\ee>0$.

$(ii)$ By contraposition.   If $(x_i)_{i=1}^\infty$ fails to be weakly null, there exist $N\in [\nn]$, $\ee>0$, and $x^*\in B_{X^*}$ such that $|x^*(x_i)|\geqslant \ee$ for all $i\in N$.  Then $$\mathcal{S}_\xi(N)\subset [N]^{<\nn}\subset \mathcal{F}_\ee^a((x_i)_{i=1}^\infty),$$   and $(x_i)_{i=1}^\infty$ is not $\xi$-weakly null.

$(iii)$ Now suppose that $(x_i)_{i=1}^\infty$ is not $\xi$-weakly null for any $\xi<\omega_1$.  Then for every $\xi<\omega_1$, there exist $N_\xi\in [\nn]$ and $k_\xi\in \nn$ such that $$\mathcal{S}_\xi(N_\xi)\subset \mathcal{F}_{1/k_\xi}^a((x_i)_{i=1}^\infty).$$  There exist $k\in \nn$ and an uncountable subset $\mathcal{U}$ of $[0, \omega_1)$ such that $k_\xi=k$ for all $\xi\in \mathcal{U}$, and $$CB(\mathcal{F}^a_{1/k}((x_i)_{i=1}^\infty)) \geqslant \underset{\xi\in \mathcal{U}}{\ \sup\ } CB(\mathcal{S}_\xi(N_\xi))=\omega_1.$$   Since $\mathcal{F}_{1/k}^a((x_i)_{i=1}^\infty)$ is hereditary, it fails to be compact, and there exists $N=(n_i)_{i=1}^\infty\in [\nn]$ such that $[N]^{<\nn}\subset \mathcal{F}_{1/k}^a((x_i)_{i=1}^\infty)$.   This means that for every $r\in \nn$, there exists $x^*_r\in B_{X^*}$ such that for every $1\leqslant i\leqslant r$, $|x^*_r(x_{n_i})|\geqslant 1/k$.    Then if $x^*$ is a weak$^*$-cluster point of $(x^*_r)_{r=1}^\infty$, $|x^*(x_n)|\geqslant 1/k$ for all $n\in N$, and $(x_i)_{i=1}^\infty$ is not weakly null.

$(iv)$ If $\xi<\zeta<\omega_1$, there exists $M=(m_i)_{i=1}^\infty\in [\nn]$ such that $\mathcal{S}_\xi(M)\subset \mathcal{S}_\zeta(N)$.  If $(x_i)_{i=1}^\infty$ is not $\zeta$-weakly null, there exist $N=(n_i)_{i=1}^\infty\in [\nn]$ and $\ee>0$ such that $\mathcal{S}_\zeta(N)\subset \mathcal{F}^a_\ee((x_i)_{i=1}^\infty)$.   Then if $r_i=n_{m_i}$ and $R=(r_i)_{i=1}^\infty$, then $\mathcal{S}_\xi(R)\subset \mathcal{S}_\zeta(N)\subset \mathcal{F}^a_\ee((x_i)_{i=1}^\infty)$, and $(x_i)_{i=1}^\infty$ is not $\xi$-weakly null.     

\end{proof}

We let $\mathfrak{L}$ denote the class of all (bounded, linear) operators between Banach spaces.  For Banach spaces $X,Y$, we let $\mathfrak{L}(X,Y)$ denote the set of all members of $\mathfrak{L}$ whose domain is $X$ and which map into $Y$.  Given a class $\mathfrak{I}$ of operators and Banach spaces $X,Y$, we let $\mathfrak{I}(X,Y)=\mathfrak{I}\cap \mathfrak{L}(X,Y)$.  We say a class $\mathfrak{I}$ has the \emph{ideal property} if for any Banach spaces $W,X,Y,Z$ and any operators $C\in \mathfrak{L}(W,X)$, $B\in \mathfrak{I}(X,Y)$, $A\in \mathfrak{L}(Y,Z)$, $ABC\in \mathfrak{I}$.    We say $\mathfrak{I}$ is an \emph{operator ideal} provided that \begin{enumerate}[(i)]\item $I_\mathbb{K}\in \mathfrak{I}$, \item for each pair $X,Y$ of Banach spaces, $\mathfrak{I}(X,Y)$ is a linear subspace of $\mathfrak{L}(X,Y)$, \item $\mathfrak{I}$ has the ideal property. \end{enumerate} Given a class $\mathfrak{J}$ of operators, we let $\text{Space}(\mathfrak{J})$ denote the class of Banach spaces $X$ such that $I_X\in \mathfrak{J}$, called the \emph{space ideal} of $\mathfrak{J}$.    We obey the established convention that for a given class denoted by a fraktur letter ($\mathfrak{A}, \mathfrak{B}, \mathfrak{I}$, etc.), the same sans serif letter ($\textsf{A}, \textsf{B}, \textsf{I}$, etc.) will denote the associated space ideal.

We recall that, given a class of operators $\mathfrak{I}$, $\complement \mathfrak{I}$ denotes the complement of the class $\mathfrak{I}$. Given an ideal $\mathfrak{I}$, the symbol $\mathfrak{I}^\text{sur}$ denotes the ideal of all operators $A\in \mathfrak{L}(X,Y)$ such that there exist a Banach space $W$ and a quotient map $q:W\to X$ such that $Aq\in \mathfrak{I}(W,Y)$. The symbol $\mathfrak{I}^\text{dual}$ denotes the class of all operators $A\in \mathfrak{L}(X,Y)$ such that $A^*\in \mathfrak{I}(Y^*, X^*)$.  Given two classes   $\mathfrak{I}, \mathfrak{J}$ of operators, $\mathfrak{I}\circ \mathfrak{J}$ denotes the class of all operators $C:X\to Y$ such that there exist a Banach space $E$ and operators $B\in \mathfrak{J}(X,E)$, $A\in \mathfrak{I}(E,Y)$ such that $C=AB$. The symbol $\mathfrak{I}\circ \mathfrak{J}^{-1}$ denotes the class of all operators $A:X\to Y$ between Banach spaces $X,Y$ such that for any Banach space $W$ and any $B\in \mathfrak{J}(W,X)$, $AB\in \mathfrak{I}(W,Y)$.   The symbol $\mathfrak{I}^{-1}\circ \mathfrak{J}$ denotes the class of operators $B:X\to Y$ such that for any Banach space $Z$ and any operator $A\in \mathfrak{I}(Y,Z)$, $AB\in \mathfrak{J}$.  

Given an operator ideal $\mathfrak{I}$, we say $\mathfrak{I}$ is \begin{enumerate}[(i)]\item \emph{closed} provided that for all Banach spaces $X,Y$, $\mathfrak{I}(X,Y)$ is norm closed in $\mathfrak{L}(X,Y)$, \item \emph{injective} if for any Banach spaces $X,Y,Z$, any $A:X\to Y$, and any isometric (equivalently, isomorphic) embedding $j:Y\to W$ such that $jA\in \mathfrak{I}$, $A\in \mathfrak{I}$, \item \emph{surjective} if for any Banach spaces $W,X,Y$, any $A:X\to Y$, and any quotient map (equivalently, surjection) $q:W\to X$ such that $Aq\in \mathfrak{I}$, $A\in \mathfrak{I}$, \item \emph{idempotent} if $\mathfrak{I}=\mathfrak{I}\circ \mathfrak{I}$, \item \emph{symmetric} if $\mathfrak{I}^\text{dual}=\mathfrak{I}$. \end{enumerate}

We will let $\mathfrak{W}$, $\mathfrak{V}$, and $\mathfrak{K}$ denote the classes of weakly compact, completely continuous, and compact operators, respectively.  The symbol  $\mathfrak{W}_\infty$ denotes the class of all operators $A:X\to Y$ such that $\overline{AB_X}$ lies in the closed, absolutely convex hull of a weakly null sequence in $Y$.  If $A:X\to Y$ is an operator and $\xi$ is a countable ordinal, we say $A$ is $\xi$-\emph{weakly compact} (or $\xi$-\emph{Banach-Saks}) provided that for any bounded sequence $(x_n)_{n=1}^\infty$, there exists a subsequence $(x_n)_{n\in M}$ of $(x_n)_{n=1}^\infty$ such that $(Ax_n)_{n\in M}$ is $\xi$-weakly convergent to a member of $Y$.  We let $\mathfrak{W}_\xi$ denote the class of all operators which are $\xi$-weakly compact and $\textsf{W}_\xi$ the class of all spaces which are $\xi$-Banach-Saks.  We let $\mathfrak{W}_{\omega_1}$ denote the class of weakly compact operators.

For any $0\leqslant \xi\leqslant \omega_1$, we say $A:X\to Y$ is $\xi$-\emph{completely continuous} provided that every $\xi$-weakly null  sequence in $X$ is sent by $A$ to a norm null  sequence in $Y$. We say a Banach space $X$ is $\xi$-\emph{Schur} if $I_X$ is $\xi$-completely continuous.   Given an operator $A:X\to Y$, we let $\textsf{v}(A)$ denote the minimum countable ordinal $\xi$ such that $A$ is not $\xi$-completely continuous if such an ordinal $\xi$ exists, and we let $\textsf{v}(A)=\omega_1$ otherwise.  For a Banach space $X$, we let $\textsf{v}(X)=\textsf{v}(I_X)$.   Given $0\leqslant \xi\leqslant \omega_1$, we let $\mathfrak{V}_\xi$ denote the class of all operators such that $\textsf{v}(A)\geqslant \xi$ and we let $\textsf{V}_\xi$ denote the space ideal associated with $\mathfrak{V}_\xi$.

We now recall the definition of an $\ell_1^\xi$-\emph{spreading model}.  For $0<\xi<\omega_1$, we say $(x_i)_{i=1}^\infty$ is an $\ell_1^\xi$-\emph{spreading model} if it is bounded and there exists $\ee>0$ such that $\mathcal{S}_\xi\subset \mathfrak{F}_\ee^\sigma((x_i)_{i=1}^\infty)$.  

\begin{rem}\upshape We deduce the following from Theorem \ref{convex unc}. For $0<\xi<\omega_1$, $A:X\to Y$ is $\xi$-completely continuous if and only if for any weakly null sequence $(x_i)_{i=1}^\infty\subset X$ such that $\inf_i \|Ax_i\|>0$, $(x_i)_{i=1}^\infty$ (equivalently, such that $(Ax_i)_{i=1}^\infty$ is normalized) has a subsequence which is an $\ell_1^\xi$-spreading model.

Similarly, for $0<\xi<\omega_1$, a Banach space $X$ is $\xi$-Schur if and only if any normalized, weakly null sequence has a subsequence which is an $\ell_1^\xi$-spreading model.

\end{rem}

\begin{proposition} \begin{enumerate}[(i)]\item $\mathfrak{V}=\mathfrak{V}_{\omega_1}$. \item $\mathfrak{L}=\mathfrak{V}_0$. \item For any $0\leqslant \xi\leqslant \zeta\leqslant \omega_1$, $\mathfrak{V}_\zeta\subset \mathfrak{V}_\xi$. \item For any $\xi<\omega_1$, there exists an operator $A:X\to Y$ such that $\textsf{\emph{v}}(A)=\xi$, whence for any $0\leqslant \xi<\zeta\leqslant \omega_1$, $\mathfrak{V}_\zeta\subsetneq \mathfrak{V}_\xi$. \end{enumerate} 

\label{easy prop}
\end{proposition}

\begin{proof}$(i)$ Suppose $A:X\to Y$ is an operator.  Since the set of weakly null sequences in $X$ coincide with the set of sequences which are $\xi$-weakly null for some $\xi<\omega_1$, then $A$ is completely continuous if and only if every sequence in this set is sent by $A$ to a norm null sequence if and only if $A$ is $\xi$-completely continuous for every $\xi<\omega_1$ if and only if $\textsf{v}(A)=\omega_1$.

$(ii)$ Since the $0$-weakly null sequences are simply the norm null sequences, every operator is $0$-completely continuous.

$(iii)$ For $0\leqslant \xi\leqslant \zeta\leqslant \omega_1$, any $\xi$-weakly null sequence is $\zeta$-weakly null.  Therefore if $A$ is $\zeta$-completely continuous, it sends $\zeta$-weakly null sequences to norm null sequences, and therefore sends $\xi$-weakly null sequences to norm null sequences, and is therefore $\xi$-completely continuous. 

$(iv)$ For convenience, we postpone the proof of $(iv)$ until the end of Section $4$.

\end{proof}

\begin{theorem}  If $K$ is any compact, Hausdorff space, then for  any Banach space $F$, \begin{align*} \mathfrak{W}_1(C(K), F) & =  \mathfrak{W}(C(K), F) = \mathfrak{V}(C(K), F)  = \mathfrak{V}_1(C(K), F). \end{align*}

\label{gds}
\end{theorem}

\begin{proof} It is a result of Grothendieck \cite{Gr} that $\mathfrak{V}(C(K), F)=\mathfrak{W}(C(K), F)$. It is a result of Diestel and Seifert \cite{DS} that $\mathfrak{W}_1(C(K), F)=\mathfrak{W}(C(K), F)$.    Obviously $\mathfrak{V}(C(K), F)\subset \mathfrak{V}_1(C(K), F)$. Finally, Pe\l czy\'{n}ski \cite{Pelc} proved that if $A:C(K)\to F$ is not weakly compact, then there exists a closed subspace $E$ of $C(K)$ which is isometric to $c_0$ such that $A|_E$ is an isomorphic embedding into $F$. Since the canonical $c_0$ basis is $1$-weakly null, $A$ cannot be $1$-completely continuous, whence $\mathfrak{V}_1(C(K), F)\subset \mathfrak{V}(C(K), F)$.

\end{proof}

\begin{proposition} Let $E_1<E_2<\ldots$ be subsets of $\nn$.   Given a hereditary family $\mathfrak{F}\subset [\nn]^{<\nn}$, let $\mathfrak{F}(E)=\{F\in [\nn]^{<\nn}: \varnothing\neq \mathfrak{F}\cap \prod_{i\in F} E_i\}$, where  $\prod_{i\in \varnothing}E_i=\{\varnothing\}$.   Then for any ordinal $\xi$ and any hereditary family $\mathfrak{F}$, $$\mathfrak{F}(E)^\xi\subset \mathfrak{F}^\xi(E).$$  In particular, $CB(\mathfrak{F}(E))\leqslant CB(\mathfrak{F})$. 

\label{descent speed}

\end{proposition}

\begin{proof} By induction. The $\xi=0$ case is clear.    Suppose $\mathfrak{F}(E)^\xi\subset \mathfrak{F}^\xi(E)$ and $F\in \mathfrak{F}(E)^{\xi+1}$.   This means there exist $F<r_1<r_2<\ldots $ such that $F\cup (r_j)\in \mathfrak{F}(E)^\xi\subset \mathfrak{F}^\xi(E)$ for all $j\in \nn$.  For each $j$, there exist $G_j\in \prod_{i\in F}E_i$ and $n_j\in E_{r_j}$ such that $G_j\cup (n_j)\in \mathfrak{F}^\xi$.  Since $\prod_{i\in F}E_i$ is finite,   there exist an infinite subset $M$ of $\nn$ and $G\in \prod_{i\in F}E_i$ such that $G_j=G$ for all $j\in M$. Since $n_j<n_k$ for all $j,k\in M$ with $j<k$, it follows that $G\in \mathfrak{F}^{\xi+1}$, and $F\in \mathfrak{F}^{\xi+1}(E)$.

Now suppose $\xi$ is a limit ordinal and $\mathfrak{F}(E)^\zeta\subset \mathfrak{F}^\zeta(E)$ for all $\zeta<\xi$.   Fix $F\in \mathfrak{F}(E)^\xi$ and note that for each $\zeta<\xi$, there exists $G_\zeta\in \mathfrak{F}^\zeta\cap \prod_{i\in F}E_i$. Since $\prod_{i\in F}E_i$ is finite, there exist $G\in \prod_{i\in F}E_i$ and  a subset $M$ of $[0, \xi)$ with $\sup M=\xi$ such that $G_\zeta=G$ for all $\zeta\in M$. From this it follows that $G\in \cap_{\zeta\in M} \mathfrak{F}^\zeta=\mathfrak{F}^\xi$, whence $F\in \mathfrak{F}^\xi(E)$.

The last statement follows from the first.

\end{proof}

\begin{lemma} Fix $\xi<\omega_1$. If $X$ is a Banach space and $(x_i)_{i=1}^\infty\subset X$ is such that $$CB(\mathfrak{F}_\ee^a((x_i)_{i=1}^\infty))\leqslant \omega^\xi$$ for every $\ee>0$, then the operator $\Phi:\ell_1\to X$ given by $\Phi\sum_{i=1}^\infty a_i e_i=\sum_{i=1}^\infty a_i x_i$ lies in $\mathfrak{W}_\xi(\ell_1, X)$.

\label{phi}
\end{lemma}

\begin{proof} If $\xi=0$, then $(x_i)_{i=1}^\infty$ is norm null and $\Phi\in \mathfrak{K}(\ell_1, X)=\mathfrak{W}_0(\ell_1, X)$.  Now assume $\xi>0$.  By \cite[Proposition $4.5$]{BCFW}, if $\Phi$ does not lie in $\mathfrak{W}_\xi(\ell_1, X)$, then there exists $(u_i)_{i=1}^\infty\subset B_{\ell_1}$ such that $(\Phi u_i)_{i=1}^\infty$ is an $\ell_1^\xi$  spreading model.  By passing to an appropriate difference sequence of a subsequence, scaling, and perturbing, we may assume $(u_i)_{i=1}^\infty \subset B_{\ell_1}$ is a block sequence in $\ell_1$ and $\ee>0$ is such that $$\mathcal{S}_\xi\subset \mathfrak{F}^\sigma_\ee((\Phi u_i)_{i=1}^\infty).$$    In particular, $CB(\mathfrak{F}^\sigma_\ee((\Phi u_i)_{i=1}^\infty)>\omega^\xi$.

Now if $E_1<E_2<\ldots$ are such that $u_i\in B_{\ell_1}\cap \text{span}\{e_j: j\in E_i\}$, it follows that $$\mathfrak{F}^\sigma_\ee((\Phi u_i)_{i=1}^\infty )\subset \{F: \exists (t_i)_{i\in F}\in \mathfrak{F}^a_\ee((x_i)_{i=1}^\infty)\cap \prod_{i\in F} E_i\},$$ and by Proposition \ref{descent speed}, $$\omega^\xi< CB(\mathfrak{F}^\sigma_\ee((\Phi u_i)_{i=1}^\infty )) \leqslant CB(\mathfrak{F}_\ee^a((x_i)_{i=1}^\infty)) \leqslant \omega^\xi,$$ a contradiction.

\end{proof}

Given $M=(m_n)_{n=1}^\infty\in [\nn]$ and a collection $\mathfrak{F}\subset [\nn]^{<\nn}$, we recall $\mathfrak{F}(M)=\{(m_n)_{n\in E}: E\in \mathfrak{F}\}$, where we agree to the convention that $(m_n)_{n\in \varnothing}=\varnothing$.

\begin{lemma} Fix a countable ordinal $\xi$ and let $\mathfrak{E}\subset [\nn]^{<\nn}$ be a hereditary family.  The following are equivalent. \begin{enumerate}[(i)] \item There exists $M\in[\nn]$ such that $\mathcal{S}_\xi(M)\subset \mathfrak{E}$. \item There exists $M\in [\nn]$ such that for every $N\in [M]$, $CB(\mathfrak{E}\cap [N]^{<\nn})\geqslant \omega^\xi+1$. \item For every regular family $\mathcal{G}\subset [\nn]^{<\nn}$  with  $CB(\mathcal{F})\leqslant \omega^\xi+1$, there exists $M\in [\nn]$ such that $\mathcal{G}(M)\subset \mathfrak{E}$.\end{enumerate}
\label{mst}

\end{lemma}

\begin{proof}$(i)\Rightarrow (ii)$ Suppose $\mathcal{S}_\xi(M)\subset \mathfrak{E}$. Then for any $N\in [M]$, $$CB(\mathfrak{E}\cap [N]^{<\nn})\geqslant CB(\mathcal{S}_\xi(M)\cap [N]^{<\nn})\geqslant CB(\mathcal{S}_\xi(N))=\omega^\xi+1.$$   Here we are using the facts that $\mathcal{S}_\xi(N)\subset \mathcal{S}_\xi(M)\cap [N]^{<\nn}$ and $\mathcal{S}_\xi(N)$ is homeomorphic to $\mathcal{S}_\xi$ and therefore has Cantor-Bendixson index $\omega^\xi+1$.

$(ii)\Rightarrow (iii)$ Suppose that $M\in [\nn]$ is such that $CB(\mathfrak{E}\cap [N]^{<\nn})\geqslant \omega^\xi+1$ for all $N\in [M]$.   If $\xi=0$, then there exists $N\in [\nn]$ such that $(n)\in \mathfrak{E}$ for all $n\in N$, otherwise $\mathfrak{E}$ is finite  and $CB(\mathfrak{E})\leqslant 1$. In this case, $\mathcal{S}_0(N)\subset \mathfrak{E}$. For any regular family $\mathcal{G}$ with $CB(\mathcal{G})\leqslant 1+1=2$, $\mathcal{G}\subset \mathcal{S}_0$, and $\mathcal{G}(N)\subset \mathcal{S}_0(N)\subset \mathfrak{E}$.   

Now suppose $0<\xi$.  Fix a sequence $\mathcal{F}_n$ of regular families with $CB(\mathcal{F}_n)\uparrow \omega^\xi$.   By replacing $\mathcal{F}_i$ with $\cup_{j=1}^i \mathcal{F}_j$, we may also assume that $\mathcal{F}_1\subset \mathcal{F}_2\subset \ldots$.  We now recursively choose $M\supset N_1\supset N_2\supset \ldots$ such that for each $i\in \nn$, either $\mathcal{F}_i\cap [N_i]^{<\nn}\subset \mathfrak{E}$ or $\mathfrak{E}\cap [N_i]^{<\nn}\subset \mathcal{F}_i$, which we may do by Theorem \ref{gasp}.    Now for each $i\in \nn$, $$\omega^\xi+1\leqslant CB(\mathfrak{E}\cap [N_i]^{<\nn})$$ and $CB(\mathcal{F}_i)<\omega^\xi$, so the second inclusion cannot hold. From this it follows that $\mathcal{F}_i\cap [N_i]^{<\nn}\subset \mathfrak{E}$ for all $i\in \nn$.  Now fix $n_1<n_2<\ldots$, $n_i\in N_i$, and let $N=(n_i)_{i=1}^\infty$ and $\mathcal{F}=\{\varnothing\}\cup \{E: \varnothing\neq E\in \mathcal{F}_{\min E}\}$.  Then $\mathcal{F}$ is regular with $CB(\mathcal{F})=\omega^\xi +1$.  We claim that $\mathcal{F}(N)\subset \mathfrak{E}$.  Indeed, fix $\varnothing \neq E\in \mathcal{F}$ and let $i=\min E$. Then $E\in \mathcal{F}_i$, and $N(E)\in \mathcal{F}_i\cap [N_i]^{<\nn}\subset \mathfrak{E}$.

Now for any regular $\mathcal{G}$ with $CB(\mathcal{G})\leqslant \omega^\xi+1$, there exists $L=(l_i)_{i=1}^\infty\in [\nn]$ such that $\mathcal{G}(L)\subset \mathcal{F}$.    Then if $j_i=n_{l_i}$ and $J=(j_i)_{i=1}^\infty$, $$\mathcal{G}(J)=\mathcal{G}(L)(N) \subset \mathcal{F}(N)\subset \mathfrak{E}.$$

$(iii)\Rightarrow (i)$ This follows from the fact that $CB(\mathcal{S}_\xi)=\omega^\xi+1$.

\end{proof}

\begin{corollary} For $0<\xi\leqslant \omega_1$, if $A:X\to Y$ fails to be $\xi$-completely continuous, then there exists a $\xi$-weakly null sequence $(x_i)_{i=1}^\infty\subset B_X$ such that the map $\Phi:\ell_1\to X$ given by $\Phi\sum_{i=1}^\infty a_ie_i=\sum_{i=1}^\infty a_ix_i$ lies in $\mathfrak{W}_\xi(\ell_1,X)$ and $(Bx_i)_{i=1}^\infty$ is bounded away from zero.

\label{divad}
\end{corollary}

\begin{proof} If $\xi=\omega_1$, this is well-known.  Indeed, if $(x_i)_{i=1}^\infty\subset B_X$ is weakly null, the map $\Phi:\ell_1\to X$ given by $\Phi \sum_{i=1}^\infty a_i e_i=\sum_{i=1}^\infty a_i x_i$ is weakly compact.

Suppose $\xi<\omega_1$.     There exists a $\xi$-weakly null sequence $(u_i)_{i=1}^\infty$ such that $(Bu_i)_{i=1}^\infty$ is bounded away from zero. Now by Lemma \ref{mst}, for every $M\in[\nn]$ and $\ee>0$,  there exists $N\in [M]$ such that $CB(\mathfrak{F}_\ee((u_i)_{i=1}^\infty)\cap [N]^{<\nn})<\omega^\xi$.   We may then recursively select $N_1\supset N_2\supset \ldots$ such that for all $n\in\nn$, $$CB(\mathfrak{F}_{1/n}((x_i)_{i=1}^\infty) \cap [N_n]^{<\nn})<\omega^\xi.$$   Fix $n_1<n_2<\ldots$, $n_i\in \nn$, and let $x_i=u_{n_i}$.  Then for all $\ee>0$, $$CB(\mathfrak{F}_\ee((x_i)_{i=1}^\infty))<\omega^\xi.$$  We conclude by Lemma \ref{phi}.

\end{proof}

\begin{theorem} Fix $1\leqslant \xi\leqslant \omega_1$.  \begin{enumerate}[(i)]\item For any $0<\xi<\omega_1$, $\mathfrak{W}_\xi^{\text{\emph{dual dual}}}=\mathfrak{W}_\xi$ and $\mathfrak{W}_\xi^\text{\emph{dual}}\neq \mathfrak{W}_\xi$. \item $\mathfrak{V}_\xi= \mathfrak{K}\circ \mathfrak{W}_\xi^{-1}= \mathfrak{W}_\infty\circ \mathfrak{W}_\xi^{-1}$. \item  $\mathfrak{V}_\xi^\text{\emph{dual}}= (\mathfrak{W}_\xi^\text{\emph{dual}})^{-1}\circ \mathfrak{K}$. \item  $((\mathfrak{W}_\xi^\text{\emph{dual}})^{-1}\circ \mathfrak{K})^\text{\emph{dual}}\subsetneq\mathfrak{V}_\xi$.  

\end{enumerate}

\end{theorem}

\begin{proof} $(i)$ It was shown in \cite{BC} that if $A\in \mathfrak{W}_\xi$, then $A$ factors through a Banach space $Z$ such that $I_Z\in \mathfrak{W}_\xi$.   Since this space $Z$ is reflexive, $A^{**}$ also factors through $Z$, and $A^{**}\in \mathfrak{W}_\xi$.   Conversely, if $A^{**}$ factors through $Z$ with $I_Z\in \mathfrak{W}_\xi$, then so does $A$.  This yields that $\mathfrak{W}_\xi^\text{dual dual}=\mathfrak{W}_\xi$.   But for any ordinal $\xi$, there exists a reflexive Banach space $T_\xi$ such that $T^*_\xi$ has no $\ell_1^1$-spreading model and $T_\xi$ has an $\ell_1^\xi$ spreading model (the higher order Tsirelson spaces serve as such examples).    Therefore $I_{T_\xi}\notin \mathfrak{W}_\xi$ while $I^*_{T_\xi}\in \mathfrak{W}_1$.

$(ii)$ It is quite clear that $\mathfrak{V}_\xi\subset \mathfrak{K}\circ \mathfrak{W}_\xi^{-1}$.    Indeed, suppose $W,X,Y$ are Banach spaces,  $A\in \mathfrak{V}_\xi(X,Y)$, and $B\in \mathfrak{W}_\xi(W,X)$.  If $(w_n)_{n=1}^\infty$ is bounded, then there exists a subsequence $(u_n)_{n=1}^\infty$ of $(w_n)_{n=1}^\infty$ such that $(Bu_n)_{n=1}^\infty$ is $\xi$-convergent to some $x\in X$, and $(ABu_n)_{n=1}^\infty$ must be norm convergent to $Ax$.  Therefore $AB$ is compact.

Now since $\mathfrak{K}\subset \mathfrak{W}_\infty$, $\mathfrak{K}\circ \mathfrak{W}_\xi^{-1}\subset \mathfrak{W}_\infty\circ \mathfrak{W}_\xi^{-1}$.

To prove that $\mathfrak{W}_\infty\circ \mathfrak{W}_\xi^{-1}\subset \mathfrak{V}_\xi$, we modify a result of Johnson, Lillemets,  and Oja from \cite{JLO}.   Seeking a contradiction, suppose that there exist  Banach spaces $X,Y$ and $A\in \mathfrak{W}_\infty \circ \mathfrak{W}_\xi^{-1}(X,Y)$ such that $A\in \complement \mathfrak{V}_\xi$.   Then there exist $(x_n)_{n=1}^\infty$ and $\ee>0$ such that $\Phi:\ell_1\to X$ given by $\Phi\sum_{n=1}^\infty a_ne_n=\sum_{n=1}^\infty a_n x_n$ lies in $\mathfrak{W}_\xi$ and $\inf_n \|Ax_n\|>\ee$.    It follows from \cite{BC} that there exists a Banach space $Z$ with $I_Z \in \mathfrak{W}_\xi$ such that $\Phi$ factors through $Z$.  Suppose that $\Phi_1:\ell_1\to Z$, $J:Z\to X$ are such that $\Phi=J\Phi_1$. Then $AJ\in \mathfrak{W}_\infty(Z,Y)$.   Then there exists a weakly null sequence $(y_n)_{n=1}^\infty$ in $Y$ such that $AJB_Z\subset \Phi_2 B_{\ell_1}$, where $\Phi_2:\ell_1\to Y$ is given by $\Phi_2\sum_{n=1}^\infty a_ne_n=\sum_{n=1}^\infty a_n y_n$. Let $\overline{\Phi}_2:\ell_2/\ker(\Phi_2)\to Y$ be the injective associate of $\Phi_2$.  Then we consider the operator $\overline{\Phi}_2^{-1}AJ:Z\to \ell_1/\ker(\Phi_2)$, which is a well-defined, bounded operator. We argue as in \cite{JLO} to deduce that $\ell_1/\ker(\Phi_2)$ has the Schur property, since $\ell_1/\ker(\Phi_2)=W^*$ for a closed subspace $W$ of $c_0$. Now since $Z$ is reflexive and $\ell_1/\ker(\Phi_2)$ has the Schur property, every  operator from $Z$ into $\ell_1/\ker(\Phi_2)$ is compact.  Therefore $\overline{\Phi}_2^{-1}AJ$ is compact, and so must be $AJ=\overline{\Phi}_2\overline{\Phi}_2^{-1}AJ$.  But since $AJ$ is clearly not compact, we reach the necessary contradiction.

$(iii)$ First suppose that $X,Y$ are Banach spaces and  $A\in \mathfrak{V}_\xi^\text{dual}(X,Y)$.  We wish to show that $A\in ((\mathfrak{W}_\xi^\text{dual})^{-1}\circ \mathfrak{K})(X,Y)$.   To that end, fix a Banach space $Z$ and $B\in \mathfrak{W}_\xi^\text{dual}(Y,Z)$. Then by $(ii)$, $(BA)^*=A^*B^*\in \mathfrak{K}(Z^*, X^*)$.  Since by Schauder's theorem, $\mathfrak{K}=\mathfrak{K}^\text{dual}$, $BA\in \mathfrak{K}(X,Z)$.   Since $Z$ and $B\in \mathfrak{W}_\xi^\text{dual}(Y,Z)$ were arbitrary, $A\in (\mathfrak{W}_\xi^\text{dual})^{-1}\circ \mathfrak{K}$.    

Now suppose that $A\in \complement \mathfrak{V}_\xi^\text{dual}$.  Then by Corollary \ref{divad}, there exists a weakly null sequence $(y^*_n)_{n=1}^\infty\subset Y^*$ such that $\Phi:\ell_1\to Y^*$ given by $\Phi\sum_{i=1}^\infty a_ie_i=\sum_{i=1}^\infty a_i y^*_i$ lies in $\mathfrak{W}_\xi(\ell_1, Y^*)$ and $\inf_n \|y^*_n\|>0$.     Define $\Phi_*:Y\to c_0$ by $\Phi_*(y)=(y^*_n(y))_{n=1}^\infty$ and note that $\Phi_*^*=\Phi$.      However, $\Phi_* A:X\to c_0$ is not compact.   From this it follows that $A\in \complement ((\mathfrak{W}_\xi^\text{dual})^{-1}\circ \mathfrak{K})$.

$(iv)$ It is quite obvious that $\mathfrak{V}_\xi\supset \mathfrak{V}_\xi^\text{dual dual}$.    That is, if $A^{**}:X^{**}\to Y^{**}$ is $\xi$-completely continuous, so is $A:X\to Y$.    From this and $(iii)$, $$\mathfrak{V}_\xi\supset \mathfrak{V}_\xi^\text{dual dual} = (\mathfrak{W}_\xi^\text{dual}\circ \mathfrak{K})^\text{dual}.$$   Thus we must show that $\mathfrak{V}_\xi\not\subset \mathfrak{V}_\xi^\text{dual dual}$.  We will show the stronger statement that $$\mathfrak{V}\not\subset \mathfrak{V}_1^\text{dual dual}.$$   Indeed, this follows from the well known fact that $\ell_1$ is a Schur space, while $\ell_1^{**}$ contains a copy of $\ell_2$, and is therefore not $1$-Schur.

\end{proof}

\begin{corollary} For any $1\leqslant \xi\leqslant \omega_1$, $\mathfrak{V}_\xi$ is a closed, injective operator ideal. Furthermore, $\mathfrak{V}_\xi$ is not symmetric,  surjective, or idempotent.

\label{willits}

\end{corollary}

In the following proof, each of our examples is a classical one. The content of the proof is not in the novelty of the examples, but in quantifying those classical examples. 

\begin{proof} The proof that $\mathfrak{V}_\xi$ is a closed, injective ideal is nearly identical to the proof that $\mathfrak{V}$ is an ideal, replacing weakly null sequences in the domain spaces with $\xi$-weakly null sequences.    

Neither $c_0$ nor $\ell_\infty$ lies in $\textsf{V}_1$, while $\ell_1\in \textsf{V}$, whence neither of the classes $\mathfrak{V}_\xi^\text{dual}$, $ \mathfrak{V}_\xi$ is contained in the other, and $\mathfrak{V}_\xi$ is not symmetric.

Let $A:\ell_1\to c_0$ be a surjection, $q:\ell_1\to \ell_1/\ker(Q)$  the quotient map, and $\overline{A}:\ell_1/\ker(Q)\to c_0$  the injective associate of $A$. That is, $A=\overline{A}q$.  Then  $A\in \mathfrak{V}$, while $\overline{A}\in \complement \mathfrak{V}_1\subset \complement \mathfrak{V}_\xi$, so that $\mathfrak{V}_\xi$ is not surjective.

Now let $j:C[0,1]\to L_2$ be the canonical inclusion.      Suppose that $Z$ is a Banach space, $R:C[0,1]\to Z$, and $L:Z\to L_2$ are operators such that $j=RL$, $L\in \mathfrak{V}_\xi(Z, L_2)\subset \mathfrak{V}_1(Z, L_2)$, and  $R\in \mathfrak{V}_\xi(C[0,1], Z)=\mathfrak{W}_1(C[0,1], Z)$.   Then $RL\in \mathfrak{K}$, which is obviously absurd. However, since $j: C[0,1]\to L_2$ is weakly compact, it is completely continuous.   Thus $j\in \mathfrak{V}$ and $j\in \complement (\mathfrak{V}_1\circ \mathfrak{V}_1)$.

\end{proof}

\begin{rem}\upshape Note that quantitatively, we have proved more than what  is stated in Corollary \ref{willits}, and indeed we have proved the strongest possible quantitative statements.  Indeed, we have shown that $$\mathfrak{V}^\text{surj}\not\subset \mathfrak{V}_1,$$ $$\mathfrak{V}^\text{dual}\not\subset \mathfrak{V}_1,$$ $$\mathfrak{V}\not\subset \mathfrak{V}^\text{dual}_1,$$  and $$\mathfrak{V}\not\subset \mathfrak{V}_1\circ \mathfrak{V}_1.$$

\end{rem}

\section{Special convex combinations}

Let $\mathscr{P}$ denote the set of all probability measures on $\nn$.    We treat each member $\mathbb{P}$ of $\mathscr{P}$ as a function from $\nn$ into $[0,1]$, where $\mathbb{P}(n)=\mathbb{P}(\{n\})$.  We let $\text{supp}(\mathbb{P})=\{n\in \nn: \mathbb{P}(n)>0\}$.   Given a nice family $\mathcal{P}$ and a subset $\mathfrak{P}=\{\mathbb{P}_{M,n}: M\in [\nn], n\in \nn\}$ of $ \mathscr{P}$, we say $(\mathfrak{P}, \mathcal{P})$ is a \emph{probability block} provided that \begin{enumerate}[(i)]\item for any $M\in [\nn]$ and $r\in \nn$, if $N=M\setminus \cup_{i=1}^{r-1}\text{supp}(\mathbb{P}_{M,i})$, then $\mathbb{P}_{N,1}=\mathbb{P}_{M,r}$, and \item for each $M\in [\nn]$, $\text{supp}(\mathbb{P}_{M,1})=M_{\mathcal{P}, 1}$.    \end{enumerate}

\begin{rem}\upshape It follows from the definition of probability block that for any $M\in [\nn]$, $(M_{\mathcal{P}, n})_{n=1}^\infty=(\text{supp}(\mathbb{P}_{M,n}))_{n=1}^\infty$ and for any $s\in \nn$ and $M,N\in \nn$, and $r_1<\ldots <r_s$ such that $\cup_{i=1}^s \text{supp}(\mathbb{P}_{M, r_i})$ is an initial segment of $N$, then $\mathbb{P}_{N, i}=\mathbb{P}_{M, r_i}$ for all $1\leqslant i\leqslant s$.   

To see the first fact, we argue by induction that any $r\in \nn$, $\text{supp}(\mathbb{P}_{M,r})=M_{\mathcal{P}, r}$. The $r=1$ case is item (ii) above.  Now assuming that for some $1<r\in \nn$ that $\text{supp}(\mathbb{P}_{M,i})=M_{\mathcal{P}, i}$ for all $1\leqslant i<r$, let $N=M\setminus \cup_{i=1}^{r-1}\text{supp}(\mathbb{P}_{M,i})=M\setminus\cup_{i=1}^{r-1} M_{\mathcal{P}, i}$.   Then $$\mathbb{P}_{M,r}=\mathbb{P}_{N,1}=N_\mathcal{P}= M_{\mathcal{P}, r}.$$

To see the second fact, we will use the obvious fact that if $M,N\in [\nn]$, $s\in \nn$, $r_1<\ldots <r_s$ are such that $\cup_{i=1}^s M_{\mathcal{P}, r_i}$ is an initial segment of $N$, then $M_{\mathcal{P}, r_i}=N_{\mathcal{P}, i}$.    For $1\leqslant i\leqslant s$, let $K=\cup_{j=1}^{r_{s-1}-1} M_{\mathcal{P}, j}$ and $L=\cup_{j=1}^{s-1} N_{\mathcal{P}, j}$, then by property (ii) above, $\mathbb{P}_{M, r_s}=\mathbb{P}_{K, 1}$ and $\mathbb{P}_{N, s}=\mathbb{P}_{L, 1}$. Since $$\text{supp}(\mathbb{P}_{K,1})=K_\mathcal{P}=M_{\mathcal{P}, r_s}= N_{\mathcal{P}, s}=L_\mathcal{P},$$ property (i) yields that $\mathbb{P}_{K,1}=\mathbb{P}_{L,1}$, whence $$\mathbb{P}_{M, r_s}=\mathbb{P}_{K,1}=\mathbb{P}_{L,1}=\mathbb{P}_{N,s}.$$

It follows from the previous paragraphs that if $M,N\in [\nn]$ and $r_1<r_2<\ldots$ are such that $N=\cup_{i=1}^\infty \text{supp}(\mathbb{P}_{M, r_i})$, then $\mathbb{P}_{N,i}=\mathbb{P}_{M, r_i}$ for all $i\in\nn$. Indeed, under these hypotheses, for any $i\in \nn$,  $\cup_{j=1}^i \text{supp}(\mathbb{P}_{M,r_j})$ is an initial segment of $N$, so $\mathbb{P}_{N, i}=\mathbb{P}_{M, r_i}$.

\end{rem}

We recall the repeated averages hierarchy defined in \cite{AMT}.   We will recall for each $\xi<\omega_1$ a collection $\mathfrak{S}_\xi$ such that $(\mathfrak{S}_\xi, \mathcal{S}_\xi)$ is a probability block.

For any $M=(m_i)_{i=1}^\infty\in[\nn]$, we let $\mathbb{S}^0_{M,n}= \delta_{m_n}$, the Dirac measure such that $\delta_{m_n}(m_n)=1$.   

Now assume that for every $M\in \nn$ and $n\in \nn$, $\mathbb{S}^\xi_{M,n}\in \mathscr{P}$ has been defined so that, with $\mathfrak{S}_\xi=\{\mathbb{S}^\xi_{M,n}: M\in [\nn], n\in \nn\}$, $(\mathfrak{S}_\xi, \mathcal{S}_\xi)$ is a probability block.  Now fix $M\in [\nn]$ and let $p_1=\min M$.  Then let $$\mathbb{S}^{\xi+1}_{M,1}=\frac{1}{p_1}\sum_{i=1}^{p_1} \mathbb{S}^\xi_{M,i}.$$  Now assuming that $\mathbb{S}^{\xi+1}_{M, 1}, \ldots, \mathbb{S}^{\xi+1}_{M,n}$ have been defined, let $$M_{n+1}= M\setminus \cup_{i=1}^n\text{supp}(\mathbb{S}^{\xi+1}_{M,i})=M\setminus \cup_{i=1}^n M_{\mathcal{S}_{\xi+1}, i},$$ $p_{n+1}=\min M_{n+1}$, and $$\mathbb{S}^{\xi+1}_{M, n+1}= \frac{1}{p_{n+1}}\sum_{i=1}^{p_{n+1}} \mathbb{S}^\xi_{M_n, i}.$$   

Now assume $\xi$ is a limit ordinal and for every $\zeta<\xi$, every $M\in [\nn]$, and every $n\in \nn$, $\mathbb{S}^\zeta_{M,n}$ has been defined.  Now let $\xi_n\uparrow \xi$ be the sequence such that $$\mathcal{S}_\xi=\{\varnothing\}\cup \{E:\varnothing\neq E\in \mathcal{S}_{\xi_{\min E}+1}\}.$$  Fix $M\in [\nn]$. For each $n\in \nn$, let $M_n=M\setminus \cup_{i=1}^{n-1} M_{\mathcal{S}_\xi, i}$, $p_n=\min M_n$, and $$\mathbb{S}^\xi_{M,n}= \mathbb{S}^{\xi_{p_n}+1}_{M_n, 1}.$$

Suppose we have probability blocks $(\mathfrak{P}, \mathcal{P})$, $(\mathfrak{Q}, \mathcal{Q})$.  We define a collection $\mathfrak{Q}*\mathfrak{P}$ such that $(\mathfrak{Q}*\mathfrak{P}, \mathcal{Q}[\mathcal{P}])$ is a probability block.     Fix $M\in \nn$ and for each $n\in\nn$,  let $l_n=\min \text{supp}(\mathbb{P}_{M,n})$ and $L=(l_n)_{n=1}^\infty$.    We then let $$\mathbb{O}_{M,n}= \sum_{i\in L_{\mathcal{Q},n}^{-1}} \mathbb{Q}_{L,n}(l_i) \mathbb{P}_{M,i}$$   and $\mathfrak{Q}*\mathfrak{P}=\{\mathbb{O}_{M,n}: M\in [\nn], n\in \nn\}$.   We have already defined $L_{\mathcal{Q}, n}^{-1}$, but for clarity, we give an alternative description.   Given $L=(l_n)_{n=1}^\infty=(\min \text{supp}(\mathbb{P}_{M,n}))_{n=1}^\infty$, there exists a unique sequence $0=p_0<p_1<\ldots$ such that $(l_i)_{i=p_{n-1}+1}^{p_n}\in MAX(\mathcal{Q})$ for all $n\in \nn$, and $L^{-1}_{\mathcal{Q}, n}=\nn\cap (p_{n-1}, p_n]$.

\begin{proposition} $(\mathfrak{Q}*\mathfrak{P}, \mathcal{Q}[\mathcal{P}])$ is a probability block. 

\end{proposition}

\begin{proof}  Throughout the proof, we will assume that $L=(l_n)_{n=1}^\infty$ is the sequence given by $$(l_n)_{n=1}^\infty = (\min \text{supp}(\mathbb{P}_{M,n}))_{n=1}^\infty$$ for some $M\in [\nn]$. We will use the fact that for any $n\in \nn$, $\cup_{i=1}^\infty L^{-1}_{\mathcal{Q}, i}=\nn$ and $(l_i)_{i\in L^{-1}_{\mathcal{Q}, n}}=L_{\mathcal{Q}, n}=\text{supp}(\mathbb{Q}_{L,n})$.

 Fix $M\in [\nn]$ and $n\in \nn$ and note that $$\mathbb{O}_{M,n}(\nn)=\sum_{i\in L^{-1}_{L,n}}\mathbb{Q}_{L,n}(l_i)\mathbb{P}_{M,i}(\nn)=\sum_{i\in L^{-1}_{\mathcal{Q}, n}}  \mathbb{Q}_{L,n}((l_i)_{i\in L^{-1}_{\mathcal{Q}, n}})= \mathbb{Q}_{L,n}(\text{supp}(\mathbb{Q}_{L,n}))=1,$$ so $\mathbb{O}_{M,n}\in \mathscr{P}$.   

Now we will use the obvious fact that if $0=p_0<p_1<\ldots$ are such that $(l_i)_{i=p_{n-1}+1}^{p_n}= L_{\mathcal{Q}, n}$ and $\text{supp}(\mathbb{P}_{M,n})=M_{\mathcal{P}, n}$ for all $n\in \nn$, then $$\text{supp}(\mathbb{O}_{M,n})=\bigcup_{i=p_{n-1}+1}^{p_n} \text{supp}(\mathbb{P}_{M,i})= (M\setminus \cup_{i=1}^{p_{n-1}} \text{supp}(\mathbb{P}_{M,i}))_{\mathcal{Q}[\mathcal{P}], n}$$ for all $n\in\nn$.   This yields that $(\mathfrak{Q}*\mathfrak{P}, \mathcal{Q}[\mathcal{P}])$ satisfies property (ii) of a probability block

Suppose that $M,N\in [\nn]$ and $r\in \nn$ are such that $$N=M\setminus \cup_{i=1}^{r-1}\text{supp}(\mathbb{O}_{M,i})= M\setminus \cup_{i=1}^{p_{r-1}} \text{supp}(\mathbb{P}_{M,i}).$$  Let $K=(k_i)_{i=1}^\infty=(\min \text{supp}(\mathbb{P}_{M,i}))_{i=p_{r-1}+1}^\infty$.     Then $K_{\mathcal{Q}, 1}^{-1}= [1, p_r-p_{r-1}]$, $L_{\mathcal{Q}, r}^{-1}=(p_{r-1}, p_r]$,  $k_i=l_{p_{r-1}+i}$, and $\mathbb{P}_{M, p_{r-1}+i}= \mathbb{P}_{N, i}$ for all $i\in \nn$. From this it follows that   $$\mathbb{O}_{N,1}= \sum_{i\in K^{-1}_{\mathcal{Q}, 1}} \mathbb{Q}_{K, 1}(k_i) \mathbb{P}_{N, i} = \sum_{i\in L^{-1}_{\mathcal{Q}, r}} \mathbb{Q}_{L, r}(l_i)\mathbb{P}_{M,i}=\mathbb{O}_{M,r},$$  and $(\mathfrak{Q}*\mathfrak{P}, \mathcal{Q}[\mathcal{P}])$ satisfies property (i) of a probability block.

\end{proof}

\begin{lemma} Let $(\mathfrak{P}, \mathcal{P})$ be a probability block and let $\xi$ be a countable ordinal. The following are equivalent. \begin{enumerate}[(i)]\item For every $\ee>0$, $q\in \nn$, regular family $\mathcal{G}$ with $CB(\mathcal{G})\leqslant \omega^\xi+1$, and $L\in [\nn]$, there exists $M\in [L]$ such that $$\sup \{\mathbb{P}_{N,1}(E): N\in [M], E\in \mathcal{G}, \min E\leqslant q\}\leqslant \ee.$$  \item For every $\ee>0$, regular family $\mathcal{G}$ with $CB(\mathcal{G})\leqslant \omega^\xi+1$, and $L\in [\nn]$, there exists $M\in [L]$ such that $$\sup \{\sum_{i=1}^\infty \mathbb{P}_{N,i}(E): N\in [M], E\in \mathcal{G}\}\leqslant 1+\ee.$$  \item For any $\ee>0$, any regular family $\mathcal{G}$ with $CB(\mathcal{G})\leqslant \omega^\xi$, and $L\in [\nn]$, there exists $M\in [L]$ such that $$\sup \{\mathbb{P}_{N,1}(E): N\in [M], E\in \mathcal{G}\}\leqslant \ee.$$  \end{enumerate}

\label{mkay}
\end{lemma}

\begin{proof}$(i)\Rightarrow (ii)$ Fix $\ee>0$, a regular family $\mathcal{G}$ with $CB(\mathcal{G})\leqslant \omega^\xi+1$, and $L\in [\nn]$.  Fix a sequence $0=\ee_0<\ee_1<\ldots$ with $\lim_n \ee_n=\ee$.    Let $M_1=L$ and fix $m_1\in M_1$.   Now assume $M_1\supset \ldots \supset M_n$, $m_1<\ldots< m_n$, $m_i\in M_i$ have been chosen.  Now fix $M_{n+1}\in [M_n]$ such that $$\sup \{\mathbb{P}_{N, 1}(E): N\in [M_n], E\in \mathcal{G}, \min E\leqslant m_n\}\leqslant \ee_n-\ee_{n-1}$$ and fix $m_{n+1}\in M_{n+1}$ with $m_n<m_{n+1}$.   This completes the recursive construction. Let $M=(m_n)_{n=1}^\infty$.  Now fix $E\in \mathcal{G}$ and $N\in [M]$. If $E\cap \text{supp}(\mathbb{P}_{N,i})=\varnothing$ for all $i\in \nn$, $\sum_{i=1}^\infty \mathbb{P}_{N,i}(E)=0$. Assume there exists $i\in \nn$ such that $E\cap \text{supp}(\mathbb{P}_{N,i})\neq \varnothing$ and let $r$ be the minimum such $i$.     Let $N_n=N\setminus \cup_{i=1}^{n-1}\text{supp}(\mathbb{P}_{N,i})$ and note that $\mathbb{P}_{N, n}=\mathbb{P}_{N_n, 1}$ for all $n\in \nn$.  Let $m_j=\max \text{supp}(\mathbb{P}_{N,r})$. Note that for each $n\in \nn$, $N_{r+n}\in [M_{j+n}]$ and $\min E\leqslant m_j \leqslant m_{j+n-1}$, so $\mathbb{P}_{N,r+n}(E) \leqslant \ee_{j+n-1}-\ee_{j+n-2}$. From this it follows that $$\sum_{i=1}^\infty \mathbb{P}_{N,1}(E) =\mathbb{P}_{N,r}(E)+\sum_{n=1}^\infty \mathbb{P}_{N,r+n}(E) \leqslant 1+\sum_{i=1}^\infty \ee_i-\ee_{i-1}=1+\ee.$$

$(ii)\Rightarrow (iii)$ Fix $\ee>0$, a regular family $\mathcal{G}$ with $CB(\mathcal{G})\leqslant \omega^\xi$, and $L\in [\nn]$.  Let $\mathcal{V}$ denote the set of $M\in [\nn]$ such that $$\sup\{\mathbb{P}_{M,1}(E): E\in \mathcal{G} \}\leqslant \ee.$$ Note that this supremum is actually a maximum, since it suffices to take the supremum only over the finitely many sets $E\in \mathcal{G}$ with $E\subset \text{supp}(\mathbb{P}_{M,1})$.   Moreover, $\mathcal{V}$ is closed, whence there exists $M\in [L]$ such that either $\mathcal{V}\cap [M]=\varnothing$ or $[M]\subset \mathcal{V}$.    To obtain a contradiction, assume $\mathcal{V}\cap [M]=\varnothing$.   Fix $n\in \nn$ such that $n\ee>2$ and let $\mathcal{H}=\mathcal{A}_n[\mathcal{G}]$. Note that $CB(\mathcal{H})\leqslant \omega^\xi$. Indeed, if $\xi=0$, then since $CB(\mathcal{G})\leqslant 1$, $\mathcal{G}\subset \{\varnothing\}$ and $\mathcal{H}\subset \{\varnothing\}$.  If $0<\xi<\omega_1$, then since $CB(\mathcal{G})\leqslant \omega^\xi$ and $\omega^\xi$ is a limit ordinal, there exist $\zeta<\xi$ and $m\in \nn$ such that $CB(\mathcal{G})\leqslant \omega^\zeta m+1$, and $$CB(\mathcal{H})\leqslant \omega^\zeta mn+1<\omega^\xi.$$  Since $CB(\mathcal{H})\leqslant \omega^\xi$, by replacing $M$ with a further, infinite subset, we may assume $$\sup\{\sum_{i=1}^\infty \mathbb{P}_{M,i}(E): E\in \mathcal{H}\}\leqslant 2.$$  Now for each $i\in\nn$, let  $M_i=M\setminus \cup_{j=1}^{i-1}\mathbb{P}_{M,j}$ and note that $M_i\in [M]\subset \mathcal{V}$.  Therefore for each $i\in \nn$, there exists $E_i\in \mathcal{G}$ such that $\mathbb{P}_{M_i, 1}(E)>\ee$. By replacing $E_i$ with $E_i\cap \text{supp}(\mathbb{P}_{M_i, 1})$, we may assume $E_i\subset \text{supp}(\mathbb{P}_{M_i, 1})$.   Let $E=\cup_{i=1}^n E_i\in \mathcal{H}$ and note that $$2\geqslant \sum_{i=1}^\infty \mathbb{P}_{M, i}(E) = \sum_{i=1}^n \mathbb{P}_{M_i, 1}(E_i)> n\ee>2,$$ a contradiction. 

$(iii)\Rightarrow (i)$ Fix $q\in \nn$, $\mathcal{G}$ regular with $CB(\mathcal{G})\leqslant \omega^\xi+1$, $\ee>0$, and $L\in [\nn]$.   Let $$\mathcal{H}=\{E\in \mathcal{G}: q<E, (q)\cup E\in \mathcal{G}\}$$ and note that $\mathcal{H}$ is regular with $CB(\mathcal{H})\leqslant \omega^\xi$. Let $L_1=(q, \infty)\cap L$ and note that for each $E\in \mathcal{G}$ with $\min E\leqslant q$, $(q)\cup (E\cap (q, \infty))$ is a spread of a subset of $E$, so $E\cap (q,\infty)\in \mathcal{H}$ and $$\mathbb{P}_{M, 1}(E)=\mathbb{P}_{M,1}(E\cap (q, \infty))$$ for any $M\in [L_1]$.  Now we may choose $M\in [L_1]$ such that $$\sup \{\mathbb{P}_{N, 1}(E): N\in [M], E\in \mathcal{G}, E\leqslant q\}\leqslant \sup \{\mathbb{P}_{N,1}(E): N\in [M], E\in \mathcal{H}\}\leqslant \ee.$$

\end{proof}

If $(\mathfrak{P}, \mathcal{P})$ satisfies any of the three equivalent conditions from Lemma \ref{mkay}, we say $(\mathfrak{P}, \mathcal{P})$ is $\xi$-\emph{sufficient}.

Given a probability block $(\mathfrak{P}, \mathcal{P})$, a function $f:\nn\times MAX(\mathcal{P})\to \rr$,  $\ee\in \rr$, and $M\in [\nn]$,  we let $$\mathfrak{E}(f, \mathfrak{P}, \mathcal{P}, \ee)=\{M\in [\nn]: \sum_{i=1}^\infty \mathbb{P}_{M,1}(i)f(i, \text{supp}(\mathbb{P}_{M,1}))\geqslant \ee\}$$ and $$\mathfrak{G}(f, \mathcal{P}, M,\ee)=\{E\in [M]^{<\nn}: (\exists F\in MAX(\mathcal{P})\cap [M]^{<\nn})(\forall i\in E)(f(i, F)\geqslant \ee)\}.$$  Note that $\mathfrak{G}(f, \mathcal{P},M, \ee)$ is hereditary.

Let us say that a hereditary family $\mathfrak{E}$ is $\xi$-\emph{full} provided that there exists $M\in [\nn]$ such that $\mathcal{S}_\xi(M)\subset \mathfrak{E}$.

For a countable ordinal $\xi$ and a probability block $(\mathfrak{P}, \mathcal{P})$, we say $(\mathfrak{P}, \mathcal{P})$ is $\xi$-\emph{regulatory} provided that for any bounded function $f:\nn\times MAX(\mathcal{P})\to \rr$ and $\delta, \ee\in \rr$ with $\delta<\ee$, if there exists $L\in [\nn]$ such that $[L]\subset \mathfrak{E}(f, \mathfrak{P}, \mathcal{P}, \ee)$, then $\mathfrak{G}(f, \mathcal{P}, L,\delta)$ is $\xi$-full.     

The following result is a generalization of a result of Schlumprecht from \cite{Schlumprecht}. 

\begin{lemma} Let $\xi$ be a countable ordinal and let $(\mathfrak{P}, \mathcal{P})$ be a probability block.  If $(\mathfrak{P}, \mathcal{P})$ is $\xi$-sufficient, it is $\xi$-regulatory. 

\label{schlumprecht}
\end{lemma}

\begin{proof} To obtain a contradiction, assume $L\in [\nn]$, $M\in [L]$, $\delta, \ee\in \rr$, $\delta<\ee$,  $f:\nn\times MAX(\mathcal{P})\to \rr$ is bounded, $[L]\subset \mathfrak{E}(f, \mathfrak{P}, \mathcal{P}, \ee)$, and $CB(\mathfrak{G}(f, \mathcal{P},L, \delta)\cap [M]^{<\nn})\leqslant \omega^\xi$. Note that $\mathfrak{G}(f, \mathcal{P}, L, \delta)\cap [M]^{<\nn}\supset \mathfrak{G}(f, \mathcal{P}, M, \delta)$, so $CB(\mathfrak{G}(f, \mathcal{P}, M, \delta))\leqslant \omega^\xi$.   If $\xi=0$, this means there exists $m\in M$ such that for all $m< i\in M$ and every $F\in \mathcal{P} \cap [M]^{<\nn}$, $f(i, F)\leqslant \delta$.  Then if $N=(m, \infty)\cap M$, $$\ee\leqslant \sum_{i=1}^\infty \mathbb{P}_{N,1}(i)f(i, \text{supp}(\mathbb{P}_{N,1}) <\delta,$$ a contradiction.  This finishes the $\xi=0$ case.

Now assume $0<\xi$.  Then since $CB(\mathfrak{G}(f, \mathcal{P},M, \delta))\leqslant \omega^\xi$, $\mathfrak{G}(f, \mathcal{P},M, \delta)$ is hereditary and compact. Since $\omega^\xi$ is a limit ordinal, it follows that $CB(\mathfrak{G}(f, \mathcal{P},M, \delta))<\omega^\xi$.  We may fix  a regular family $\mathcal{G}$ such that $CB(\mathfrak{G}(f, \mathcal{P},M, \delta))<CB(\mathcal{G})<\omega^\xi$. By using Theorem \ref{gasp} and replacing $M$ with a further infinite subset, we may assume $\mathfrak{G}(f, \mathcal{P},M, \delta)\subset \mathcal{G}$.   Now fix $s> \|f\|_\infty$ and, by passing once more to an infinite subset of $M$, assume that $$\sup \{\mathbb{P}_{N,1}(E): E\in \mathcal{G}, N\in [M]\}< \frac{\ee-\delta}{s}.$$   Now let $E=\{n\in \text{supp}(\mathbb{P}_{M,1}): f(n, \text{supp}(\mathbb{P}_{M,1}))\geqslant \delta\}$ and note that $E\in \mathfrak{G}(f, \mathcal{P},M, \delta)\subset \mathcal{G}$, whence \begin{align*} \ee &  \leqslant \sum_{i=1}^\infty \mathbb{P}_{M,1}(i) f(i, \text{supp}(\mathbb{P}_{M,1}))\\ &  = \sum_{i\in E}\mathbb{P}_{M,1}(i) f(i, \text{supp}(\mathbb{P}_{M,1})) +\sum_{i\in \nn\setminus E} \mathbb{P}_{M,1}(i) f(i, \text{supp}(\mathbb{P}_{M,1})) \\ & \leqslant s \mathbb{P}_{M,1}(E)+ \delta \mathbb{P}_{M,1}(\nn\setminus E) \\ & <  \ee. \end{align*} 

\end{proof}

\begin{theorem} Let $\xi, \zeta$ be countable ordinals. Then if $(\mathfrak{P}, \mathcal{P})$ is $\xi$-sufficient and $(\mathfrak{Q}, \mathcal{Q})$ is $\zeta$-sufficient, $(\mathfrak{Q}*\mathfrak{P}, \mathcal{Q}[\mathcal{P}])$ is $\xi+\zeta$-sufficient, and therefore $\xi+\zeta$-regulatory.  

\label{weather chopper}
\end{theorem}

Before the proof, we isolate the following fact. 

\begin{fact} Suppose $\xi, \zeta$ are countable ordinals and $\zeta>0$.  Then for any regular family $\mathcal{G}$ with $CB(\mathcal{G})\leqslant \omega^{\xi+\zeta}$, there exist regular families $\mathcal{F}$, $\mathcal{H}$ with $CB(\mathcal{H})<\omega^\zeta$ and $CB(\mathcal{F})=\omega^\xi+1$ such that $\mathcal{G}\subset \mathcal{H}[\mathcal{F}]$.   

\label{easy}
\end{fact}

\begin{proof} If $CB(\mathcal{G})\leqslant \omega^{\xi+\zeta}$, then since $CB(\mathcal{G})$ cannot be a limit ordinal and $\omega^{\xi+\zeta}$ is a limit ordinal, $CB(\mathcal{G})<\omega^{\xi+\zeta}$. This means there exist $\eta<\zeta$ and $m\in \nn$ such that $CB(\mathcal{G})<\omega^{\xi}\omega^\eta m+1$.    Using Theorem \ref{gasp}, there exists $M=(m_n)_{n=1}^\infty \in [\nn]^{<\nn}$ such that $\mathcal{G}\cap [M]^{<\nn}\subset (\mathcal{A}_m[\mathcal{S}_\eta])[\mathcal{S}_\xi]$.    Now let $$\mathcal{F}=\{E: (m_n)_{n\in E}\in \mathcal{S}_\xi\}$$ and $$\mathcal{H}= \{E: (m_n)_{n\in E} \in \mathcal{A}_m[\mathcal{S}_\eta]\}.$$    Then $\mathcal{F}$, $\mathcal{H}$ are regular, $CB(\mathcal{F})=CB(\mathcal{S}_\xi)=\omega^\xi+1$, and $CB(\mathcal{H})=CB(\mathcal{A}_m[\mathcal{S}_\eta])=\omega^\eta m+1<\omega^\zeta$.    Furthermore, if $E\in \mathcal{G}$, then $(m_n)_{n\in E}$ is a spread of $E$, and therefore lies in $\mathcal{G}\cap [M]^{<\nn}$. From this it follows that there exist $F_1<\ldots <F_n$ such that $\varnothing\neq F_i\in \mathcal{S}_\xi$, $(\min F_i)_{i=1}^n\in \mathcal{A}_m[\mathcal{S}_\eta]$, and $(m_i)_{i\in E}=\cup_{i=1}^n F_i$.   We may then write $F_i=(m_j)_{j\in E_i}$ for some $E_1<\ldots <E_n$ with $E=\cup_{i=1}^n E_i$.  Then $(m_n)_{n\in E_i}\in \mathcal{S}_\xi$, whence $E_i\in \mathcal{F}$, and $(m_{\min E_i})_{i=1}^n\in \mathcal{A}_m[\mathcal{S}_\eta]$, so $(\min E_i)_{i=1}^n\in \mathcal{H}$.  From this it follows that $E\in \mathcal{H}[\mathcal{F}]$.

\end{proof}

\begin{proof}[Proof of Theorem \ref{weather chopper}] Let $\mathfrak{Q}*\mathfrak{P}=\{\mathbb{O}_{M,n}: M\in [\nn], n\in \nn\}$.

First suppose that $\zeta=0$.   We must show that $(\mathfrak{Q}*\mathfrak{P}, \mathcal{Q}[\mathcal{P}])$ is $\xi$-sufficient.  Fix $\ee>0$, $L\in [\nn]$, $q\in \nn$, and a regular family $\mathcal{G}$ with $CB(\mathcal{G})\leqslant \omega^\xi+1$.  There exists $M\in [L]$ such that for all $N\in [\nn]$ and $E\in \mathcal{G}$ with $\min E\leqslant q$, $\mathbb{P}_{N,1}(E)\leqslant \ee$.    Now fix any $N\in [\nn]$ and $E\in \mathcal{G}$ with $\min E\leqslant q$.   Let $K=(k_i)_{i=1}^\infty=(\min \text{supp}(\mathbb{P}_{N, i}))_{i=1}^\infty$.  If $N_n=N\setminus \cup_{i=1}^{n-1} \text{supp}(\mathbb{P}_{N,i})$, $$\mathbb{P}_{N,r}(E)=\mathbb{P}_{N_r, 1}(E)\leqslant \ee,$$ whence \begin{align*} \mathbb{O}_{N,1}(E) = \sum_{i\in K^{-1}_{\mathcal{Q}, 1}} \mathbb{Q}_{K, 1}(k_i)\mathbb{P}_{N, i}(E) \leqslant \ee \mathbb{Q}_{K,1}(\nn)= \ee.\end{align*}  This gives the $\zeta=0$ case.

Now suppose that $\zeta>0$.  Fix $\ee>0$, $L\in [\nn]$, and a regular family $\mathcal{G}$ with $CB(\mathcal{G})<\omega^{\xi+\zeta}+1$. Using Fact \ref{easy} and replacing $\mathcal{G}$ with a larger family, we may assume that $\mathcal{G}=\mathcal{H}[\mathcal{F}]$ for some regular families $\mathcal{H}$, $\mathcal{F}$ with $CB(\mathcal{F})=\omega^\xi+1$ and $CB(\mathcal{H})<\omega^\zeta$. Let $\mathcal{V}$ denote the set of those $N\in [\nn]$ such that $\sup\{\mathbb{O}_{N,1}(E): E\in \mathcal{G}\}\leqslant \ee$ and note that $\mathcal{V}$ is closed.  Then  we may fix $M\in [L]$ such that either $[M]\subset \mathcal{V}$ or $[M]\cap \mathcal{V}=\varnothing$.  Seeking a contradiction, suppose $[M]\cap \mathcal{V}=\varnothing$.  Now let us recursively choose $m_1\in M_1:= M$, $M_2\in [M_1]$, $m_1<m_2\in M_2$, $M_3\in [M_2]$, etc., such that for all $1<n\in \nn$, $$\sup \{\mathbb{P}_{N,1}(E): N\in [M_n], E\in \mathcal{F}, \min E\leqslant m_{n-1}\} \leqslant \frac{\ee}{3\cdot 2^n}.$$   By replacing $M$ with $(m_n)_{n=1}^\infty\in [M]$, we may assume that for all $N\in [M]$, all $r\in \nn$, all $E\in \mathcal{F}$ such that $\text{supp}(\mathbb{P}_{N,r})\cap E\neq \varnothing$, and all $n\in \nn$, $$\mathbb{P}_{N,n+r}(E) \leqslant \frac{\ee}{3\cdot 2^{r+n}},$$ so that $$\sum_{i=r+1}^\infty \mathbb{P}_{N,i}(E) \leqslant \frac{\ee}{3\cdot 2^r}.$$  Indeed, let $j\in \nn$ be such that $m_j=\max \text{supp}(P_{N,r})$ and note that $j\geqslant r$.   Then for any $n\in \nn$, if $N_n=N\setminus \cup_{i=1}^{r+n-1} \text{supp}(\mathbb{P}_{N,i})$, $N_n\in [M_{j+n}]$ and $\min E\leqslant m_j \leqslant m_{j+n-1}$, whence $$\mathbb{P}_{N, n+r}(E)= \mathbb{P}_{N_n, 1}(E) \leqslant \frac{\ee}{3\cdot 2^{j+n}}\leqslant \frac{\ee}{3\cdot 2^{r+n}}.$$  Then $$\sum_{i=r+1}^\infty \mathbb{P}_{N,i}(E) = \sum_{i=1}^\infty \mathbb{P}_{N,r+i}(E) \leqslant \sum_{i=1}^\infty \frac{\ee}{3\cdot 2^{r+n}}= \frac{\ee}{3\cdot 2^r}.$$

Now let $r_n=\max \text{supp}(\mathbb{P}_{M,n})$ and let $$\mathcal{I}=\{E: (r_n)_{n\in E}\in \mathcal{H}\}.$$  Note that $\mathcal{I}$ is regular and $CB(\mathcal{I})=CB(\mathcal{H})<\omega^\zeta$.   Fix $S=(s_n)_{n=1}^\infty\in [\nn]$ such that $$\sup \{\mathbb{Q}_{T,1}(E): T\in [S], E\in \mathcal{I}\} \leqslant \frac{\ee}{3}$$ and let $$N=\bigcup_{n\in S}\text{supp}(\mathbb{P}_{M, n})=\bigcup_{n=1}^\infty \text{supp}(\mathbb{P}_{M, s_n})\in [M]$$ and note that $\mathbb{P}_{N,n}=\mathbb{P}_{M, s_n}$ for all $n\in \nn$. Let $t_n=\min \text{supp}(\mathbb{P}_{N,n})$ and  $T=(t_n)_{n=1}^\infty$. Write $$\mathbb{O}_{N,1}= \sum_{i\in T_{\mathcal{Q}, 1}^{-1}} \mathbb{Q}_{T,1}(t_i)\mathbb{P}_{N,i}.$$    Since $[M]\cap \mathcal{V}=\varnothing$, $$\sup \{\mathbb{O}_{N,1}(E): E\in \mathcal{G}\}>\ee.$$   We will obtain the desired contradiction by proving that $$\sup \{\mathbb{O}_{N,1}(E): E\in \mathcal{G}\}\leqslant \ee.$$    For this, it is sufficient to show that for any $E=\cup_{j=1}^p E_j$, $F=\cup_{j=1}^q F_j$ such that $E_1<\ldots <E_p$, $F_1<\ldots <F_q$, $\varnothing\neq E_j, F_j\in \mathcal{F}$, and $(\min E_j)_{j=1}^p, (\min F_j)_{j=1}^q\in \mathcal{H}$ such that for each $1\leqslant j\leqslant p$, there exists at most one $i\in T^{-1}_{\mathcal{Q},1}$ such that $E_j\cap \text{supp}(\mathbb{P}_{N,i})\neq 0$, and for each $1\leqslant  j\leqslant q$, there exist at least two values of $i\in T^{-1}_{\mathcal{Q},1}$ such that $F_j\cap \text{supp}(\mathbb{P}_{N,i})\neq 0$, then $$\mathbb{O}_{N,1}(E) \leqslant \ee/3$$ and $$\mathbb{O}_{N,1}(F)\leqslant 2\ee/3.$$  This is sufficient to produce a contradiction, because any member of $\mathcal{G}=\mathcal{H}[\mathcal{F}]$ is the union of two sets $E,F$ having the properties described above.  Let $E=\cup_{j=1}^p E_j$, $F=\cup_{j=1}^q F_j$ be as described above.   For each $i\in T^{-1}_{\mathcal{Q}, 1}$, let $$B_i=\{j\leqslant p: E_j\cap \text{supp}(\mathbb{P}_{N,i})\neq \varnothing\}.$$   Then the sets $(B_i)_{i\in T^{-1}_{\mathcal{Q}, 1}}$ are pairwise disjoint by the properties of the sets $E_1, \ldots, E_p$.    Let $B=\{i\in T^{-1}_{\mathcal{Q}, 1}: B_i\neq \varnothing\}$. We claim that $G:=(t_i)_{i\in B}\in \mathcal{I}$, from which it will follow that $$\mathbb{O}_{N,1}(E)= \sum_{i\in B} \mathbb{Q}_{T,1}(t_i)\sum_{j\in B_i} \mathbb{P}_{N,i}(E_j) \leqslant \sum_{i\in B} \mathbb{Q}_{T,1}(t_i)= \mathbb{Q}_{T,1}(G)\leqslant \ee/3.$$Let us see why $G\in \mathcal{I}$.      For each $i\in B$, fix $j_i\in B_i$ and note that $(r_{t_i})_{i\in B}$ is a spread of $(\min  E_{j_i})_{i\in B}\in \mathcal{H}$, since $$\min E_{j_i}\leqslant \max \text{supp}(\mathbb{P}_{N,i})= \max \text{supp}(\mathbb{P}_{M, s_i}) = r_{s_i} \leqslant r_{t_i}.$$ Here we are using that $t_i=\min \text{supp}(\mathbb{P}_{M, s_i})\geqslant s_i$.  From this it follows that $(r_{t_i})_{i\in B}\in \mathcal{H}$ and $G=(t_i)_{i\in B}\in \mathcal{I}$.

Now for each $j\leqslant q$, let $i_j=\min \{i\in T^{-1}_{\mathcal{Q}, 1}: F_j\cap \text{supp}(\mathbb{P}_{N,i})\neq \varnothing\}$.    Then by the properties of the sets $F_1, \ldots, F_q$, the map $j\mapsto i_j$ is an injection from $\{1, \ldots, q\}$ into $T^{-1}_{\mathcal{Q},1}$.    Let $C=(i_j)_{j\leqslant q}$. We first claim that $H:=(t_i)_{i\in C}\in \mathcal{I}$.  Indeed, just as in the last paragraph, $(r_{t_{i_j}})_{j=1}^q$ is a spread of $(\min F_j)_{j=1}^q\in \mathcal{H}$.   Therefore, using the properties of $N\in [M]$ established in the second paragraph of the proof,  \begin{align*} \mathbb{O}_{N,1}(F) & = \sum_{j=1}^q \sum_{i_j\leqslant i\in T^{-1}_{\mathcal{Q}, 1}} \mathbb{Q}_{T,1}(t_i)\mathbb{P}_{N,i}(E_j) \\ & = \sum_{j=1}^q \Bigl[ \mathbb{Q}_{T,1}(t_{i_j}) + \sum_{i=i_j+1}^\infty \mathbb{P}_{N,i}(E_j)\Bigr] \\ & \leqslant \mathbb{Q}_{T,1}(H) + \sum_{j=1}^q \frac{\ee}{3\cdot 2^{i_j}} \\ & \leqslant \frac{\ee}{3}+ \frac{\ee}{3}=2\ee/3. \end{align*}

\end{proof}

\begin{lemma}\cite[Lemma $6.4$]{C1} Fix $K=(k_n)_{n=1}^\infty\in [\nn]$ and $\ee>0$. Then for any $L=(l_n)_{n=1}^\infty$ such that  $k_{l_n}(1+2\ee)<l_{n+1}\ee$ for all $n\in \nn$, $$\sup \Bigl\{\sum_{i=1}^\infty \mathbb{S}^\xi_{L,i}(E): (k_n)_{n\in E}\in \mathcal{S}_\xi\Bigr\} \leqslant 1+\ee$$  for any $\xi<\omega_1$. 

\label{fast growing}

\end{lemma}

The next corollary follows immediately from Lemmas \ref{schlumprecht} and  \ref{fast growing}.  

\begin{corollary} For every $\xi<\omega_1$,  $(\mathfrak{S}_\xi, \mathcal{S}_\xi)$ is $\xi$-sufficient. Therefore for every  $\xi, \zeta<\omega_1$, $(\mathfrak{S}_\zeta*\mathfrak{S}_\xi, \mathfrak{S}_\zeta[\mathfrak{S}_\xi])$ is $\xi+\zeta$-sufficient and therefore $\xi+\zeta$-regulatory.

\end{corollary}

\begin{corollary} Let $\xi, \zeta$ be countable ordinals. If $X,Y,Z$ are Banach spaces, $B:X\to Y$ has $\textsf{\emph{v}}(B)> \zeta$, and $A:Y\to Z$ has $\textsf{\emph{v}}(A)> \xi$, then $\textsf{\emph{v}}(AB)> \xi+\zeta$.

In particular, if $X$ is a Banach space with $\textsf{\emph{v}}(X)>\xi$, then $\textsf{\emph{v}}(X)>\xi+\xi$. 

\label{block of block}
\end{corollary}

\begin{proof} If $\zeta=0$, the result is trivial, so assume $\zeta>0$. If either $A$ or $B$ is completely continuous, $\textsf{v}(AB)=\omega_1>\xi+\zeta$.  Assume neither $A$ nor $B$ is completely continuous.  Let $$\mathfrak{S}_\zeta*\mathfrak{S}_\xi=\{\mathbb{P}_{M,n}: M\in [\nn], n\in \nn\}.$$ Seeking a contradiction, assume $(x_i)_{i=1}^\infty\subset X$ is weakly null, $(ABx_i)_{i=1}^\infty$ is seminormalized, but there do not exist $N\in [\nn]$ and $\delta>0$ such that $\mathcal{S}_{\xi+\zeta}(N)\subset \mathfrak{F}_\delta((x_i)_{i=1}^\infty)$.  We may assume $(x_i)_{i=1}^\infty \subset B_X$.   Since $(Bx_i)_{i=1}^\infty$ is weakly null and $(ABx_i)_{i=1}^\infty$ is seminormalized, there exist $\ee_1>0$ and $R\in [\nn]$ such that $\mathcal{S}_\xi(R)\subset \mathfrak{F}_{\ee_1}((Bx_i)_{i=1}^\infty)$. By passing to a subsequence, we may assume $R=\nn$ and $\mathcal{S}_\xi\subset \mathfrak{F}_{\ee_1}((Bx_i)_{i=1}^\infty)$.   Now for each $n\in\nn$, let $$\mathcal{V}_n=\{M\in [\nn]: \|\sum_{i=1}^\infty \mathbb{P}_{M,1}(i)x_i\|\geqslant 1/n\}.$$    We recursively select $L_1\supset L_2\supset \ldots$ such that for all $n\in \nn$, either $[L_n]\subset \mathcal{V}_n$ or $[L_n]\cap \mathcal{V}_n=\varnothing$. We first remark that for all $n\in \nn$, $[L_n]\cap \mathcal{V}_n=\varnothing$. If it were not so, then $[L_n]\subset \mathcal{V}_n$.   For all $F\in MAX(\mathcal{S}_\zeta[\mathcal{S}_\xi])$, if $M,N\in [\nn]$ are such that $F$ is a common initial segment of $M$ and $N$, then $$x_F:=\sum_{i=1}^\infty \mathbb{P}_{M,1}(i)x_i=\sum_{i=1}^\infty \mathbb{P}_{N,1}(i)x_i.$$  We may fix $x^*_F\in B_{X^*}$ such that $\text{Re\ }x^*_F(x_F)=\|x_F\|$.  Define $f:\nn\times MAX(\mathcal{S}_\zeta[\mathcal{S}_\xi])\to \rr$ by $f(i, F)=\text{Re\ }x^*_F(x_i)$.   Then $[L_n]\subset \mathfrak{F}(f, \mathfrak{S}_\zeta*\mathfrak{S}_\xi, \mathcal{S}_\zeta[\mathcal{S}_\xi], \ee_1)$, whence for any $0<\delta<\ee_1$ and any $M\in [L_n]$, $CB(\mathfrak{G}(f, \mathcal{S}_\zeta[\mathcal{S}_\xi], M, \delta))\geqslant \omega^{\xi+\zeta}+1$. From this and Lemma \ref{mst}, there would exist some $M\in [L_n]$ such that $\mathcal{S}_{\xi+\zeta}(M)\subset \mathfrak{F}_\delta((x_i)_{i=1}^\infty)$, which is a contradiction.  Therefore $[L_n]\cap \mathcal{V}_n=\varnothing$ for all $n\in\nn$.   Now fix $l_1<l_2<\ldots$, $l_n\in L_n$ and let $L=(l_n)_{n=1}^\infty$. 

Now let $y_n=\sum_{i=1}^\infty \mathbb{S}^\xi_{L,n}(i)x_i$.  Since $\mathcal{S}_\xi\subset \mathfrak{F}_{\ee_1}((Bx_i)_{i=1}^\infty)$, $\inf_n \|By_n\|\geqslant \ee_1$.  Since $(y_n)_{n=1}^\infty$ is weakly null in $X$, there exist $0<\ee<1$ and $M=(m_n)_{n=1}^\infty\in [\nn]$ such that $\mathcal{S}_\zeta(M)\subset \mathfrak{F}_\ee((y_n)_{n=1}^\infty)$. Fix $m>3/\ee$ and recursively select $m=t_1<t_2<\ldots$ such that $t_{i+1}>\min \text{supp}(\mathbb{S}^\xi_{L, m_{t_i}})$ for all $i\in \nn$.    Now let $$N=\bigcup_{n=1}^\infty \text{supp}(\mathbb{S}^\xi_{L, m_{t_n}})\in [L_m],$$ $$s_i=\min \text{supp}(\mathbb{S}^\xi_{L, m_{t_i}}),$$ and $$S= (s_i)_{i=1}^\infty.$$   We remark that $(t_i)_{i\in S^{-1}_{\mathcal{S}_\zeta, 1}\setminus \{1\}}\in \mathcal{S}_\zeta$, whence $(m_{t_i})_{i\in S^{-1}_{\mathcal{S}_\zeta, 1}\setminus \{1\}}\in \mathcal{S}_\zeta(M)\subset \mathfrak{F}_\ee((y_i)_{i=1}^\infty)$.    In order to see that $(t_i)_{i\in S^{-1}_{\mathcal{S}_\zeta, 1}\setminus \{1\}}\in \mathcal{S}_\zeta$, note that since $t_{i+1}>s_i$ for all $i\in \nn$, $(t_i)_{i\in S^{-1}_{\mathcal{S}_\zeta, 1}\setminus \{1\}}$ is a spread of $$(s_i)_{i\in S^{-1}_{\mathcal{S}_\zeta, 1}\setminus \{\max S^{-1}_{\mathcal{S}_\zeta, 1}\}}=S_{\mathcal{S}_\zeta, 1}\setminus \{\max S_{\mathcal{S}_\zeta, 1}\} \subset S_{\mathcal{S}_\zeta, 1}\in \mathcal{S}_\zeta.$$  Now since $(m_{t_i})_{i\in S^{-1}_{\mathcal{S}_\zeta, 1}\setminus \{1\}}\in \mathfrak{F}_\ee((y_i)_{i=1}^\infty)$, $$\|\sum_{i\in S^{-1}_{\mathcal{S}_\zeta, 1}\setminus \{1\}} \mathbb{S}^\zeta_{S,1}(s_i)  y_{m_{t_i}}\|\geqslant \ee \sum_{i\in S^{-1}_{\mathcal{S}_\zeta, 1}\setminus \{1\}} \mathbb{S}^\zeta_{S,1}(s_i) \geqslant \ee(1- \mathbb{S}^\zeta_{S,1}(s_1)) \geqslant \ee (1-1/\min S) \geqslant 2\ee/3,$$  and $$\|\sum_{i\in S^{-1}_{\mathcal{S}_\zeta, 1}} \mathbb{S}^\zeta_{S,1}(s_i) y_{m_{t_i}}\|\geqslant 2\ee/3 - \mathbb{S}^\zeta_{S,1}(s_1) \|y_{m_{t_1}}\| \geqslant 2\ee/3 - 1/\min S \geqslant \ee/3.$$  Here we are using the fact that for any $0<\xi<\omega_1$ and any $J\in[\nn]$, $\sup_{j\in\nn} \mathbb{S}^\xi_{J,1}(j)\leqslant 1/\min J$.   However, since $N\in [L_m]$, \begin{align*} 1/m & > \|\sum_{i=1}^\infty \mathbb{P}_{N,1}(i) x_i\| = \|\sum_{i\in S^{-1}_{\mathcal{S}_\zeta, 1}}\mathbb{S}^\zeta_{S,1}(s_i)\sum_{j=1}^\infty \mathbb{S}^\xi_{N, i}(j)x_j\| \\ & = \|\sum_{i\in S^{-1}_{\mathcal{S}_\zeta, 1}}\mathbb{S}^\zeta_{S,1}(s_i)y_{m_{t_i}}\| \geqslant \ee/3>1/m, \end{align*} a contradiction.  This contradiction yields that $\textsf{v}(AB)>\xi+\zeta$.

For the last statement, apply the first part with $X=Y=Z$, $A=B=I_X$, and $\zeta=\xi$. 

\end{proof}

\begin{rem}\upshape There are some instances in which the lower estimate provided by Corollary \ref{block of block} does not provide anything stronger than the ideal property. Corollary \ref{block of block} yields that for any $n\in\nn$, Banach spaces $X_0, \ldots, X_n$, and operators $A_i:X_i\to X_{i-1}$, $$\textsf{v}(A_1\ldots A_n)\geqslant \sup\{\zeta_1+\ldots +\zeta_n+1: \zeta_i<\textsf{v}(A_i)\}.$$  However,  if $0<\xi<\omega_1$ and $\zeta_1, \ldots, \zeta_n<\omega^\xi$, $\zeta_1+\ldots +\zeta_n+1\leqslant \omega^\xi$.  Therefore if $\textsf{v}(A_i)\leqslant \omega^\xi$ for each $1\leqslant i\leqslant n$ and $\textsf{v}(A_i)=\omega^\xi$ for at least one value of $i$, $$\textsf{v}(A_1\ldots A_n)\geqslant \sup\{\zeta_1+\ldots +\zeta_n+1: \zeta_i<\textsf{v}(A_i)\}=\omega^\xi=\max\{\textsf{v}(A_i): 1\leqslant i\leqslant n\},$$ which is the same estimate provided by the ideal property of $\mathfrak{V}_{\omega^\xi}$.

\end{rem}

\begin{corollary} For any Banach space $X$, $\textsf{\emph{v}}(X)=\omega^\gamma$ for some ordinal $\gamma\leqslant \omega_1$. 

\end{corollary}

\begin{proof} Suppose that $\xi<\textsf{v}(X)$. Then $\xi+\xi<\textsf{v}(X)$. By standard properties of ordinals, since $\textsf{v}(X)>0$, there exists an ordinal $\gamma$ such that $\textsf{v}(X)=\omega^\gamma$.   Obviously $\gamma\leqslant \omega_1$.

\end{proof}

We also obtain the following generalization and quantitative improvement of a theorem of Argyros and Gasparis.  This is a quantitataive improvement in the sense that Argyros and Gasparis showed that if $(x_i)_{i=1}^\infty$ is a sequence in $X$ such that for some $L\in[\nn]$, $$[L]\subset \{M\in[\nn]: \|\sum_{i=1}^\infty \mathbb{S}^\xi_{M,1}(i)x_i\|\geqslant \ee\},$$  then there exists $P\in [L]$ such that $$\mathcal{S}_\xi(P)\subset \mathfrak{F}_{\ee/2}((x_i)_{i=1}^\infty).$$ Under the same hypothesis, we show that for any $0<\delta<\ee$, there exists $P\in [L]$ such that $$\mathcal{S}_\xi(P)\subset \mathfrak{F}_\delta((x_i)_{i=1}^\infty).$$

\begin{theorem} Let $(x_i)_{i=1}^\infty$ be a sequence in the Banach space $X$.  Suppose that $(\mathfrak{P}, \mathcal{P})$ is $\xi$-regulatory and $CB(\mathcal{P})=\omega^\xi+1$.    Then for any $\ee>0$, exactly one of the following holds: \begin{enumerate}[(i)]\item There exist $\ee_1>\ee$ and $M\in [\nn]$ such that $\mathcal{P}(M)\subset \mathfrak{F}_{\ee_1}((x_i)_{i=1}^\infty)$. \item For every $\ee_2>\ee$ and $L\in [\nn]$, there exists $M\in [L]$ such that $$\inf \{\|\sum_{i=1}^\infty \mathbb{P}_{N,1}(i)x_i\|: N\in [M]\}\leqslant \ee_2.$$ \end{enumerate}

In particular, given a weakly null sequence $(x_i)_{i=1}^\infty$, the following are equivalent. \begin{enumerate}[(i)]\item There exist  $(\mathfrak{Q}, \mathcal{Q})$ which is $\xi$-regulatory and $L\in [\nn]$ such that  $CB(\mathcal{Q})=\omega^\xi+1$ and $$\inf\Bigl\{\|\sum_{i=1}^\infty \mathbb{Q}_{N,1}(i)x_i\|: N\in [L]\Bigr\}>0.$$ \item For any $(\mathfrak{Q}, \mathcal{Q})$ which is $\xi$-regulatory and such that $CB(\mathcal{Q})=\omega^\xi+1$,  there exists $\ee>0$ and $M\in[\nn]$ such that  $$\inf\Bigl\{\|\sum_{i=1}^\infty \mathbb{Q}_{N,1}(i) x_i\|: N\in [M]\Bigr\}\geqslant \ee.$$  \item $(x_i)_{i=1}^\infty$ fails to be $\xi$-weakly null. \end{enumerate}

\label{strm}
\end{theorem}

\begin{proof} If $\xi=0$, then $(i)$ and $(ii)$ are equivalent to the following and the negation of the following, respectively: There exists $\ee_1>\ee$ such that  $\{n\in\nn: \|x_n\|\geqslant \ee_1\}$ is infinite.

We now prove that $(i)$ and $(ii)$ are exclusive and exhaustive in the case that $0<\xi<\omega_1$.  First assume there exist $\ee_1>\ee$ and $M\in[\nn]$ such that $\mathcal{P}(M)\subset \mathfrak{F}_{\ee_1}((x_i)_{i=1}^\infty)$.    By  Proposition \ref{ironside}, there exists $N\in[\nn]$ such that for any $\varnothing\neq E\in \mathcal{S}_\xi\cap [N]^{<\nn}$, $E\setminus (\min E)\in \mathcal{P}(M)$.  Now fix $R>0$ such that $\sup_n \|x_n\|<R$ and $\delta>0$ such that $(1-\delta)\ee_1-R\delta>\ee$ and fix $\ee<\ee_2< (1-\delta)\ee_1-R\delta$.    Since $CB(\mathcal{A}_1)=2<\omega^\xi$, using the definition of $\xi$-sufficient and by passing to a further infinite subset of $N$ and relabeling, we may assume $\sup \{\mathbb{P}_{Q,1}(i): i\in\nn, Q\in [N]\}\leqslant \delta$.    Now fix $Q\in [N]$ and let $E=\text{supp}(\mathbb{P}_{Q,1})$ and $m=\min E$.     Note that $E\setminus (m)\in \mathcal{P}(M)$, whence $$\|\sum_{i\in E\setminus (m)} \mathbb{P}_{Q,1}(i)x_i\|\geqslant \ee_1 \mathbb{P}_{Q,1}(E\setminus (m)) \geqslant (1-\delta)\ee_1.$$   Now $$\|\sum_{i\in E} \mathbb{P}_{Q,1}(i)x_i\|\geqslant \|\sum_{i\in E\setminus (m)} \mathbb{P}_{Q,1}x_i\|- \mathbb{P}_{Q,1}(m)R \geqslant (1-\delta)\ee_1 -\delta R.$$  Thus $$\inf\{\|\sum_{i=1}^\infty \mathbb{P}_{Q,1}(i)x_i\|: Q\in [N]\}>\ee_2>\ee,$$ and $(ii)$ does not hold.

Now assume that $(ii)$ does not hold.  This means there exist $\ee_2>\ee_1>\ee$ and $L\in[\nn]$ such that for every $M\in[\nn]$, $$\inf \{\|\sum_{i=1}^\infty \mathbb{P}_{N,1}(i)x_i\|: N\in[M]\}>\ee_2.$$  Now for each $F\in MAX(\mathcal{P})$, fix a functional $x^*_F\in B_{X^*}$ such that $$\|\sum_{i=1}^\infty \mathbb{P}_{N,1}(i)x_i\|=\text{Re }x^*_N(\sum_{i=1}^\infty \mathbb{P}_{N,1}(i)x_i),$$   where $N$ is any member of $[\nn]$ which has $F$ as an initial segment.  Note that the vector $x_F:=\sum_{i=1}^\infty \mathbb{P}_{N,1}(i)x_i$ is independent of the choice of $N$ which extends $F$.    Now define $f:\nn\times MAX(\mathcal{P})\to \rr$ by $f(n, F)=\text{Re\ }x^*_F(x_n)$.   Then $$[M]\subset \mathfrak{E}(f, \mathfrak{P}, \mathcal{P}, \ee_2).$$   Since $(\mathfrak{P}, \mathcal{P})$ is $\xi$-regulatory and $\ee_1<\ee_2$, there exists $R\in[\nn]$ such that $$\mathcal{P}(R)\subset \mathfrak{G}(f, \mathcal{P}, M, \ee_1)\subset \mathfrak{F}_{\ee_1}((x_i)_{i=1}^\infty).$$

The second statement follows immediately from the first.

\end{proof}

The next easy consequence will be a convenient characterization for us to use in the final section. 

\begin{proposition} Let $\xi$ be a countable ordinal and let $(\delta_n)_{n=1}^\infty$ be a positive sequence with $\lim_n \delta_n=0$. Then for an operator $A:X\to Y$, $A$ fails to be $\xi$-completely continuous if and only if there exist a bounded sequence $(x_i)_{i=1}^\infty\subset X$ and $M\in[\nn]$ such that for all $N\in[M]$ and $n\in\nn$, $\|\sum_{i=1}^\infty \mathbb{S}^\xi_{N,n}(i)x_i\|\leqslant \delta_n$ and $\inf_n \|Ax_i\|>0$.  

\label{and another thing}
\end{proposition}

\begin{proof}

First suppose that $A$ fails to be $\xi$-completely continuous.   Then there exists $(x_i)_{i=1}^\infty\subset X$ which is $\xi$-weakly null and $\inf_n \|Ax_n\|>0$.    For each $\ee>0$, let $\mathcal{V}_\ee$ denote the set of those $M\in[\nn]$ such that $\|\sum_{i=1}^\infty \mathbb{S}^\xi_{M,1}(i)x_i\|>\ee$. Note that $\mathcal{V}_\ee$ is closed, and by Theorem \ref{strm}, there do not exist $\ee>0$ and $M\in[\nn]$ such that $[M]\subset \mathcal{V}_\ee$.    From this and the Ramsey theorem, for any $L\in[\nn]$ and $\ee>0$,  there exists $M\in[L]$ such that $[M]\cap \mathcal{V}_\ee=\varnothing$.    Now for any $L\in[\nn]$, we may recursively select $L\supset L_1\supset \ldots$ such that for each $n\in\nn$, $[L_n]\cap \mathcal{V}_{\delta_n}=\varnothing$.   Now fix $m_1<m_2<\ldots$ and $M=(m_n)_{n=1}^\infty$.  For $N\in[M]$ and $n\in\nn$, let $R=\cup_{i=1}^{n-1} \text{supp}(\mathbb{S}^\xi_{N,i})\in [L_n]$.  Then $$\|\sum_{i=1}^\infty \mathbb{S}^\xi_{N,n}(i)x_i\|=\|\sum_{i=1}^\infty \mathbb{S}^\xi_{R,1}(i)x_i\|\leqslant \delta_n.$$   This gives one implication.

Now suppose that $(x_i)_{i=1}^\infty$, $M\in[\nn]$ are such that for all $N\in[M]$ and $n\in\nn$, $\|\sum_{i=1}^\infty \mathbb{S}^\xi_{N,n}(i)x_i\| \leqslant \delta_n$ and $\inf_n \|Ax_n\|>0$.  We claim that $(x_n)_{n\in M}$ is $\xi$-weakly null, which will yield that $A$ is not $\xi$-completely continuous.  Write $M=(m_n)_{n=1}^\infty$.  If $\xi=0$, then the hypothesis yields that for all $n\in\nn$, $\|x_{m_n}\|\leqslant \delta_n$, and $(x_n)_{n\in M}$ is norm null, and therefore $0$-weakly null.    Now suppose that $0<\xi<\omega_1$.     Suppose that $\ee>0$, $R\in[\nn]$ are such that $$\mathcal{S}_\xi(R)\subset \mathfrak{F}_\ee((x_{m_n})_{n=1}^\infty).$$  Then with $s_n=m_{r_n}$, $$\mathcal{S}_\xi(S)\subset \mathfrak{F}_\ee((x_n)_{n=1}^\infty).$$  Now arguing as in Theorem \ref{strm}, we may select $N\in[S]\subset [M]$ such that $$\inf\{\|\sum_{i=1}^\infty \mathbb{S}^\xi_{Q,1}(i)x_i\|: Q\in [N]\}\geqslant \ee/2.$$   Now fix $n\in\nn$ such that $\ee/2>\delta_n$ and let $Q=N\setminus \cup_{i=1}^{n-1}\text{supp}(\mathbb{S}^\xi_{N,i})$.     Then $\mathbb{S}^\xi_{Q,1}=\mathbb{S}^\xi_{N,n}$, and $$\ee/2 \leqslant \|\sum_{i=1}^\infty \mathbb{S}^\xi_{Q,1}(i)x_i\|=\|\sum_{i=1}^\infty \mathbb{S}^\xi_{N,n}(i)x_i\|\leqslant \delta_n,$$ a contradiction.

\end{proof}

We now give the promised proof of Proposition \ref{easy prop}$(iv)$. 

\begin{proof}[Proof of Proposition \ref{easy prop}$(iv)$] It is sufficient to prove that for any $0<\xi<\omega_1$, there exists an operator $A$ with $\textsf{v}(A)=\xi$.    If $\xi$ is a successor, say $\xi=\gamma+1$, let $A:X_\gamma\to c_0$ be the canonical inclusion of the Schreier space $X_\gamma$ into $c_0$. We recall that the norm of $X_\gamma$ is given by $$\|x\|_{X_\gamma}= \max\{ \|Ex \|_{\ell_1}: E\in \mathcal{S}_\gamma\}.$$    We will first show that $A$ is $\gamma$-completely continuous, for which it is sufficient to show that if $(x_n)_{n=1}^\infty$  is a block sequence such that $\inf_n \|x_n\|_{c_0}\geqslant 1$, then $(x_i)_{i=1}^\infty$ is an $\ell_1^\gamma$-spreading model.  To that end, suppose $(x_n)_{n=1}^\infty$ is such a block sequence and for each $n\in\nn$, fix $m_n\in \text{supp}(x_n)$ such that $|e^*_{m_n}(x_n)|\geqslant 1$.   Then for any $E\in \mathcal{S}_\gamma$, $(m_n)_{n\in E}$ is a spread of $E$ and lies in $\mathcal{S}_\gamma$.  From this it follows that for any $E\in \mathcal{S}_\gamma$ and scalars $(a_n)_{n\in E}$, $$\|\sum_{n\in E} a_n x_n\|_{X_\gamma} \geqslant \|(m_i)_{i\in E}\sum_{n\in E}a_nx_n\|_{\ell_1} \geqslant \sum_{n\in E}|a_n|.$$  This yields that $\textsf{v}(A)>\gamma$, and $\textsf{v}(A)\geqslant \gamma+1=\xi$.  Now it is well-known that the canonical $X_\gamma$ basis has no subsequence which is an $\ell_1^{\gamma+1}$ spreading model (and we will prove a more general fact in the limit ordinal case), while the canonical $X_\gamma$ basis is normalized in $c_0$.  Therefore the canonical $X_\gamma$ basis is $\gamma+1=\xi$-weakly null in $X_\gamma$ and the image is not norm null in $c_0$, yielding that $\textsf{v}(A)\leqslant \xi$.

Now suppose that $\xi$ is a limit ordinal and fix any sequence $(\gamma_n)_{n=1}^\infty\subset [0, \xi)$ with $\sup_n \gamma_n=\xi$.     Let $X$ be the completion of $c_{00}$ with respect to the norm $\|\cdot\|_X=  \sum_{n=1}^\infty 2^{-n} \|\cdot\|_{X_{\gamma_n}}$ and let $A:X\to c_0$ be the canonical inclusion.  For each $n\in\nn$, let $I_n:X\to X_{\gamma_n}$ and $J_n:X_{\gamma_n}\to c_0$ be the canonical inclusions.   Then since $A=J_nI_n$, the ideal property yields that $$\textsf{v}(A)\geqslant \sup_n \textsf{v}(J_nI_n) \geqslant \sup_n \textsf{v}(I_n)=\xi.$$   We will show that the canonical basis of $X$ has no subsequence which is an $\ell_1^\xi$-spreading model. Since the image of this basis under $A$ is normalized in $c_0$, this will yield that $\textsf{v}(A)\leqslant \xi$ and finish the proof.   Seeking a contradiction, suppose that $M\in [\nn]$ and $m\in\nn$ are such that $$2^{-m}\leqslant \inf \{\|x\|_X: E\in \mathcal{S}_\xi, x\in \text{co}(e_{m_i}: i\in E)\}.$$  Let $$\mathcal{G}=\bigcup_{j=1}^{m+1} \{E: (m_i)_{i\in E}\in \mathcal{S}_{\gamma_i}\}.$$   Note that $$CB(\mathcal{G})= \max_{1\leqslant i\leqslant m+1} CB(\mathcal{S}_{\gamma_i})<\omega^\xi.$$  This means there exists $N\in[\nn]$ such that $$\sup\{\mathbb{S}^\xi_{N,1}(E): E\in \mathcal{G}\}< 2^{-m-1}.$$   Now let $x=\sum_{i=1}^\infty \mathbb{S}^\xi_{N,1}(i) e_{m_i}$ and note that, by our contradiction hypothesis, $\|x\|_X\geqslant 2^{-m}$.      Now for each $n\in\nn$, fix $E_n\in \mathcal{S}_{\gamma_n}$ such that $\|E_nx\|_{\ell_1}=\|x\|_{X_{\gamma_n}}$.    Now for each $n\leqslant m+1$, let $F_n=(i: m_i\in E_n)\in \mathcal{G}$, so \begin{align*} \|x\|_X & =\sum_{n=1}^\infty 2^{-n} \|E_nx\|_{\ell_1} \leqslant \sum_{n=1}^{m+1} 2^{-n}\mathbb{S}^\xi_{N,1}(F_n) + \sum_{n=m+2}^\infty 2^{-n}\|x\|_{\ell_1} \\ & <2^{-m-1}\sum_{n=1}^\infty 2^{-n} + \sum_{n=m+2}^\infty 2^{-n} = 2^{-m},\end{align*} a contradiction.

\end{proof}

\section{An interesting construction}

In this section, we perform a contruction of a transfinite modification of spaces studied by Argyros and Beanland \cite{AB}, which were themselves modifications of a space constructed by Odell and Schlumprecht \cite{OS} of a Banach space which admits no $c_0$ or $\ell_p$ spreading model.

In this section, for a subset $F$ of $\nn$, we let $F$ denote the projection from $c_{00}$  onto $\{e_i: i\in F\}$. That is, for $x=\sum a_i e_i$, $Fx=\sum_{i\in F} a_ie_i$. 

Throughout this section, $N_1, N_2, \ldots$ will be infinite, pairwise disjoint subsets of $\nn$ such that $\cup_{i=1}^\infty N_i=\nn$.   For a Banach space $X$ with normalized, bimonotone basis $(x_i)_{i=1}^\infty$, we let $q:c_{00}\to X$ be the linear extension of the map $q e_n=x_i$ for $n\in N_i$.    We then define the norm $\|\cdot\|_{E_X}$ on $c_{00}$ by $$\|x\|_{E_X}= \sup \{\|qIx\|: I\subset \nn\text{\ an interval}\}.$$   We let $E_X$ be the completion of $c_{00}$ with respect to this norm, and we let $q:E_X\to X$ denote the linear extension of $q$ to $E_X$. We note that for any $n_1<n_2<\ldots$, if $n_i\in N_i$, then $\|\sum_{i=1}^\infty a_i e_{n_i}\|_{E_X}=\|\sum_{i=1}^\infty a_ix_i\|$ for all $(a_i)_{i=1}^\infty\in c_{00}$.  We recall the following properties from this construction, shown in \cite{AB}.  

\begin{proposition} If $(u_i)_{i=1}^\infty$ is a weakly null sequence in $X$ and $(v_i)_{i=1}^\infty$ is a weakly null, bounded block sequence in $E_X$ such that $qv_i=u_i$ for all $i\in\nn$, then $(v_i)_{i=1}^\infty$ is weakly null in $E_X$.  Moreover, for any $0<\zeta<\omega_1$, if $(v_i)_{i=1}^\infty$ does not have a subsequence which is an $\ell_1^\zeta$ spreading model, then $(u_i)_{i=1}^\infty$ does not have a subsequence which is an $\ell_1^\zeta$ spreading model. 

\label{facts}

\end{proposition}

    Fix $0<\vartheta<1$ and let $\vartheta_n=\vartheta/2^n$.    For each $0<\xi<\omega_1$, let $\xi_n\uparrow \omega^\xi$ be the sequence such that $$\mathcal{S}_{\omega^\xi}= \{\varnothing\}\cup \{E: \exists n\leqslant E\in \mathcal{S}_{\xi_n}\}.$$   If $E$ is a Banach space such that the canonical $c_{00}$ basis is a normalized, bimonotone Schauder basis for $E$, then we define the space $Z_\xi(E)$ to be the completion of $c_{00}$ under the unique norm satisfying $$\|x\|_{Z_\xi(E)}=\max\Bigl\{\|x\|_E, \bigl(\sum_{n=1}^\infty \|x\|_n^2\bigr)^{1/2}\Bigr\},$$ where for each $n\in \nn$, $$\|x\|_n= \sup\Bigl\{\vartheta_n \sum_{i=1}^d \|I_ix\|_{Z_\xi(E)}: I_1<\ldots <I_d, \varnothing\neq I_i\text{\ an interval}, (\min I_i)_{i=1}^d\in \mathcal{S}_{\xi_n}\Bigr\}.$$

\begin{proposition} Suppose $E$ is a Banach space for which the canonical $c_{00}$ basis is a normalized, bimonotone basis.  Then for a  block sequence $(w_i)_{i=1}^\infty$ in $B_{Z_\xi(E)}$, $$\inf\{\|z\|_{Z_\xi(E)}: F\in \mathcal{S}_{\omega^\xi}, z\in \text{\emph{co}}(w_i: i\in F)\} \leqslant \max\Bigl\{\Bigl( \frac{1+\vartheta^2}{2}\Bigr)^{1/2}, \sup_i \|w_i\|_E\Bigr\}.$$

\label{nas}
\end{proposition}

\begin{proof} Suppose that $$\inf \{\|z\|_{Z_\xi(E)}: F\in \mathcal{S}_{\omega^\xi}, z\in \text{co}(w_i: i\in F)\}>\theta>\max\Bigl\{\Bigl(\frac{1+\vartheta^2}{2}\Bigr)^{1/2}, \sup_i \|w_i\|_E\Bigr\}.$$       For each $i,n\in\nn$, let $u_i(n)=\|w_i\|_n$. Let $u_i=(u_i(n))_{n=1}^\infty\in B_{\ell_2}$, and, by passing to subsequence, we may assume $(u_i)_{i=1}^\infty$ is weakly convergent to $u=(u(n))_{n=1}^\infty\in B_{\ell_2}$.  Now note that since $\mathcal{S}_1\subset \mathcal{S}_{\omega^\xi}$, \begin{align*} \theta & <  \underset{n\to \infty}{\ \lim\ }\underset{l_1}{\lim\inf}\ldots \underset{l_n}{\lim\inf} \|n^{-1}\sum_{i=1}^n w_{l_i}\|_{Z_\xi(E)} \\ &  \leqslant \underset{n\to \infty}{\ \lim\ }\underset{l_1}{\lim}\ldots \underset{l_n}{\lim} \|n^{-1}\sum_{i=1}^n u_{l_i}\|_{\ell_2} \\ & = \|u\|_{\ell_2}. \end{align*}  We may fix $k\in\nn$ such that $\theta^2 <\sum_{i=1}^k u(n)^2$.  Note that since $2\theta^2 > 1+\vartheta^2$, $$(1+ \vartheta^2 -\theta^2)^{1/2}<\theta,$$ so we may fix $\ee>0$ such that $\ee k+ (\vartheta^2 +1-\theta^2)^{1/2}<\theta$.   By passing to a subsequence, we may assume that for all $1\leqslant n\leqslant k$, $\|w_i\|_n \leqslant u(n)+\ee$ and $\sum_{j=1}^k \|w_j\|_n^2>\theta^2$.  Now let $m_i=\max \text{supp}(w_i)$ and let $$\mathcal{G}=\{F: (m_j)_{j\in F}\in \cup_{i=1}^k \mathcal{S}_{\xi_i}\}.$$  Fix $N\in[\nn]$ such that $$\sup \{\mathbb{S}^{\omega^\xi}_{N,1}(G): G\in \mathcal{G}\}\leqslant \ee.$$   We may do this, since $CB(\mathcal{G})=\omega^{\xi_k}+1<\omega^{\omega^\xi}+1$.

Now let $F=N_{\mathcal{S}_{\omega^\xi}}$ and let  $w=\sum_{j\in F} \mathbb{S}^{\omega^\xi}_{N,1}(j) w_j$.    Fix $1\leqslant n\leqslant k$ and non-empty intervals $I_1<\ldots <I_d$ such that $(\min I_i)_{i=1}^d \in \mathcal{S}_{\xi_n}$.    Let $G$ denote the set of those $j\in F$ such that there exist two distinct values $s,t\in \{1, \ldots, d\}$ such that $I_s w_j, I_t w_j\neq 0$. For each $j\in G$, let $t_j=\min\{s\in \{1, \ldots, d\}: I_s w_j\neq 0\}$.   Then $G\in \mathcal{G}$, since $(m_j)_{j\in G}$ is a spread of $(\min \text{supp} I_{t_j})_{j\in G}$, so that $$\vartheta_n \sum_{i=1}^d \|I_i \sum_{j\in G} \mathbb{S}^{\omega^\xi}_{N,1}(j)w_j\|_{Z_\xi(E)} \leqslant \sum_{j\in G} \mathbb{S}^{\omega^\xi}_{N,1}(j) <\ee.$$    Moreover, for each $j\in F\setminus G$, $$\vartheta_n \sum_{i=1}^d \|I_i w_j\|_{Z_\xi(E)} \leqslant  \vartheta_n.$$  Thus $$\vartheta_n\sum_{i=1}^d \|I_j w\|_{Z_\xi(E)} \leqslant \vartheta_n+\ee.$$   From this it follows that for each $1\leqslant n\leqslant k$, $\|w\|_n\leqslant \vartheta_n+\ee$.     Now since $\theta^2 < \sum_{n=1}^k \|w_j\|_n^2$ for each $j\in \nn$, it follows that $$1-\theta^2 > \sum_{n=k+1}^\infty \|w_j\|_n^2$$ for all $j\in\nn$, whence by the triangle inequality, $$\sum_{n=k+1}^\infty \|w\|_n^2 \leqslant 1-\theta^2.$$    Then $$\|w\|_E\leqslant \sup_i \|w_i\|_E<\theta,$$   \begin{align*}  \Bigl(\sum_{n=1}^\infty \|w\|_n^2 \Bigr)^{1/2} & =\Bigl(\sum_{n=1}^k \|w\|_n^2 + \sum_{n=k+1}^\infty \|w\|_n^2\Bigr)^{1/2} \\ & < \ee k + \Bigl( \sum_{n=1}^k \vartheta_n^2 +  1-\theta^2  \Bigr)^{1/2}\\ & < \ee k + \Bigl(\vartheta^2 + 1-\theta^2\Bigr)^{1/2}< \theta,\end{align*} and $$\|w\|_{Z_\xi(E)} = \max\bigl\{\|w\|_E, \bigl(\sum_{n=1}^\infty \|w\|_n^2\bigr)^{1/2}\bigr\}<\theta,$$ a contradiction.

\end{proof}

\begin{proposition}  Suppose that $(u_i)_{i=1}^\infty\subset B_Z$ is an $\ell_1^{\omega^\xi}$ spreading model in a Banach space $Z$ and $T:Z\to Y$ is a bounded linear operator such that there does not exist a subsequence of $(u_i)_{i=1}^\infty$ the image of which under $T$ is an $\ell_1^{\omega^\xi}$ spreading model in $Y$.  Then for any $0<\delta<1$, there exists an absolutely convex block $(u_i')_{i=1}^\infty$ of $(u_i)_{i=1}^\infty$ such that \begin{enumerate}[(i)]\item $\|Tu_i'\|\leqslant \delta$ for all $i\in \nn$, \item for all $F\in \mathcal{S}_{\omega^\xi}$ and all scalars $(a_i)_{i\in F}$, $$(1-\delta) \sum_{i\in F}|a_i|\leqslant \|\sum_{i\in F} a_i u_i'\| \leqslant \sum_{i\in F}|a_i|.$$   \end{enumerate}

\label{BCM}
\end{proposition}

\begin{proof}[Sketch] This fact is essentially contained in \cite{BCM}, and is a transfinite version of the James non-distortability of $\ell_1$ argument.     Fix $\ee>0$ such that for all $F\in \mathcal{S}_{\omega^\xi}$ and scalars $(a_i)_{i\in F}$, $$\ee\sum_{i\in F}|a_i|\leqslant \|\sum_{i\in F} a_iu_i\|\leqslant \sum_{i\in F}|a_i|.$$   

Now either for every $n\in\nn$ and $M\in [\nn]$, there exists $(n_i)_{i=1}^\infty\in [M]$ such that for all $F\in \mathcal{S}_{\xi_n}$ and all scalars $(a_i)_{i\in F}$, $$\delta/\ee \sum_{i\in F}|a_i|\leqslant \|T\sum_{i\in F}a_i u_{n_i}\|,$$ or there exist $n\in \nn$ and  $M\in [\nn]$ such that for all $(n_i)_{i=1}^\infty\in [M]$, there exist $F\in \mathcal{S}_{\xi_n}$ and scalars $(a_i)_{i\in F}$ with $\sum_{i\in F}|a_i|=1$ and $\|T\sum_{i\in F} a_i u_{n_i}\|<\ee/\delta$.    The first alternative cannot hold, otherwise we could extract a subsequence of $(u_i)_{i=1}^\infty$ the image of which under $T$ is an $\ell_1^{\omega^\xi}$ spreading model.   Since the second alternative must hold, there must exist some $n\in\nn$ and $M\in[\nn]$ as in the second alternative.   We may then fix $N\in[M]$ such that  $\mathcal{S}_{\omega^\xi}[\mathcal{S}_{\xi_n}](N)\subset \mathcal{S}_{\omega^\xi}$,  $F_1<F_2<\ldots$, $F_i\in \mathcal{S}_{\xi_n}$, and scalars $(d_i)_{i\in \cup_{j=1}^\infty F_j}$ such that $\sum_{i\in F_j}|d_i|=1$ and $\|T\sum_{i\in F_j} d_iu_{n_i}\|<\delta/\ee$.  Then with $u^0_j=\sum_{i\in F_j} d_i u_{n_i}$, for every $F\in \mathcal{S}_{\omega^\xi}$ and all scalars $(a_i)_{i\in F}$, $$\ee\sum_{i\in F}|a_i| \leqslant \|\sum_{i\in F} a_i u^0_i\|\leqslant \sum_{i\in F}|a_i|.$$   Now let $b_0=\ee$ and $c_0=1$.

Now applying a similar dichotomy to that in the previous paragraph, we recursively select absolutely convex blocks $(u^n_i)_{i=1}^\infty$ of $(u_i)_{i=1}^\infty$ and numbers $b_n\leqslant c_n$ such that for all $n\in \nn\cup \{0\}$, \begin{enumerate}[(i)]\item $b_n\leqslant b_{n+1}\leqslant c_{n+1}\leqslant c_n$, \item $(u^{n+1}_i)_{i=1}^\infty$ is an absolutely convex block of $(u_i^n)_{i=1}^\infty$, \item $b_n/c_n\leqslant \ee^{1/2^n}$, \item for all $F\in \mathcal{S}_{\omega^\xi}$ and all scalars $(a_i)_{i\in F}$, $$b_n\sum_{i\in F} |a_i|\leqslant \|\sum_{i\in F} a_i u^n_i\|\leqslant c_n\sum_{i\in F} |a_i|.$$ \end{enumerate}  Now we fix $n\in\nn$ such that $\ee^{1/2^n}>1-\delta$ and let $u_i'=c_n^{-1}u^n_i$.

\end{proof}

\begin{corollary} Let $(v_i)_{i=1}^\infty$ be a bounded block sequence in $Z_\xi(E)$. \begin{enumerate}[(i)]\item $\inf\{\|z\|_{Z_\xi(E)}: z\in \text{\emph{co}}(z_i: i\in\nn)\} \leqslant \sup \{\|v_i\|_E: i\in \nn\}.$  \item If $(v_i)_{i=1}^\infty$ is weakly null in $E$, it is weakly null in $Z_\xi(E)$. \item If $(v_i)_{i=1}^\infty$ does not have a subsequence which is an $\ell_1^{\omega^\xi}$ spreading model in $E$, then $(v_i)_{i=1}^\infty$ does not have a subsequence which is an $\ell_1^{\omega^\xi}$ spreading model in $Z_\xi(E)$.  \end{enumerate}

\label{conn}
\end{corollary}

\begin{proof}$(i)$ Seeking a contradiction, assume $$\inf\{\|z\|_{Z_\xi(E)}: z\in \text{co}(v_i: i\in\nn)\}=a$$ and $$\sup \{\|v_i\|_E: i\in\nn\}=b<a.$$   For each $n\in\nn$, let $$a_n=\{\|z\|_{Z_\xi(E)}: z\in \text{co}(v_i: i\geqslant n)\}\geqslant a>b.$$    We may select $\alpha>\sup_n a_n$ and $n\in\nn$ such that $$\frac{a_n}{\alpha}> \max\bigl\{\bigl(\frac{1+\vartheta^2}{2}\bigr)^{1/2}, b/\alpha\bigr\}.$$  We may then select a convex block sequence $(w_i)_{i=1}^\infty$ of $(v_i)_{i=n}^\infty$ such that $\|w_i\|_{Z_\xi(E)}\leqslant \alpha$ and $$a_n \leqslant \sup\{\|w\|_{Z_\xi(E)}: w\in \text{co}(w_i:i\in\nn)\}.$$  Then $(\alpha^{-1} w_i)_{i=1}^\infty \subset B_{Z_\xi(E)}$ and $$\inf \{\|w\|_{Z_\xi(E)}: w\in \text{co}(\alpha^{-1} w_i:i\in\nn)\} \geqslant \frac{a_n}{\alpha} > \max\bigl\{\bigl(\frac{1+\vartheta^2}{2}\bigr)^{1/2}, \sup_i \|\alpha^{-1} w_i\|_E\bigr\},$$ contradicting Proposition \ref{nas}.

$(ii)$ If $(v_i)_{i=1}^\infty$ is not weakly null in $Z_\xi(E)$, then after passing to a subsequence, we may assume $$\inf\{\|z\|_{Z_\xi(E)}: z\in \text{co}(v_i:i\in\nn)\}=\ee>0.$$  Since $(v_i)_{i=1}^\infty$ is weakly null in $E$, there is a convex block sequence $(w_i)_{i=1}^\infty$ of $(v_i)_{i=1}^\infty$ such that $\sup_i \|w_i\|_E<\ee$. But since $\inf\{\|z\|_{Z_\xi(E)}: z\in \text{co}(w_i:i\in\nn)\}\geqslant \ee$, this contradicts $(i)$.

$(iii)$ Seeking a contradiction, assume that, after scaling, passing to a subsequence, and relabeling, $(v_i)_{i=1}^\infty$ is an $\ell_1^{\omega^\xi}$ spreading model in $B_{Z_\xi(E)}$.  Fix $0<\delta<1/2$ such that $1-\delta >\bigl(\frac{1+\vartheta^2}{2}\bigr)^{1/2}$. Then by Proposition \ref{BCM} applied with $Z=Z_\xi(E)$, $Y=E$, and $T:Z_\xi(E)\to E$ the formal inclusion,  there exists an absolutely convex block sequence $(w_i)_{i=1}^\infty$ of $(v_i)_{i=1}^\infty$ such that \begin{enumerate}[(i)]\item $\|w_i\|_E\leqslant \delta<1-\delta$ for all $i\in \nn$, \item $\inf\{\|w\|: F\in \mathcal{S}_{\omega^\xi}, w\in \text{co}(w_i: i\in F)\} \geqslant 1-\delta> \bigl(\frac{1+\vartheta^2}{2}\bigr)^{1/2}$. \end{enumerate} But these two things contradict Proposition \ref{nas}.

\end{proof}

\begin{corollary} Let $X$ be a Banach space with normalized, bimonotone basis $(x_i)_{i=1}^\infty$ and let $E$ be a Banach space such that the canonical $c_{00}$ basis  is a normalized, bimonotone basis for $E$. \begin{enumerate}[(i)]\item $X$ is a quotient of $Z_\xi(E_X)$. \item If $X$ is not $\omega^\xi$-Schur, $Z_\xi(E_X)$ is not $\omega^\xi$-Schur. \item  $\textsf{\emph{v}}(Z_\xi(E))\geqslant \omega^\xi$. \item If $\textsf{\emph{v}}(X)\leqslant \omega^\xi$, then $\textsf{\emph{v}}(Z_\xi(E_X))=\omega^\xi$.  \end{enumerate}

\label{beans}

\end{corollary}

\begin{proof}$(i)$ Fix $x=\sum_{i=1}^d a_i x_i\in S_X$.   Fix $n_{11}<n_{12}<\ldots <n_{1d}<n_{21}<n_{22}<\ldots <n_{2d}<\ldots$ such that $n_{ji}\in N_i$.  For each $j\in\nn$, let $y_j=\sum_{i=1}^d a_i e_{n_{ji}}$.  Then $\|y_j\|_{Z_\xi(E_X)}\leqslant d$ and $\|y_j\|_{E_X}=1$ for all $j\in \nn$.   Furthermore, if $Q:Z_\xi(E_X)\to X$ is the continuous, linear extension of the map $q$ sending $e_n$ to $x_i$ for $n\in N_i$, then $Qy_j=x$ for all $j\in\nn$.   By Corollary \ref{conn}, since $Q$ is norm $1$,  $$1 \leqslant \inf\{\|z\|_{Z_\xi(E_X)}: z\in Z_\xi(E_X), Qz=x\} \leqslant \inf\{\|z\|_{Z_\xi(E)}: z\in \text{co}(y_j: j\in\nn)\}\leqslant 1.$$   Thus $Q$ is a quotient map.

We remark that in the proof above, if $n\in\nn$ is fixed, we could have chosen $n_{11}>n$.  From this it follows that for any $x=\sum_{i=1}^d a_ix_i\in S_X$ and any $n\in \nn$, there exists $z\in c_{00}\cap 2B_{Z_\xi(E_X)}$ with $\min \text{supp}(z)>n$ such that $Qz=x$.

$(ii)$ If $X$ is not $\omega^\xi$-Schur, then there exists a normalized, weakly null block sequence $(u_i)_{i=1}^\infty \subset S_X$ which has no subsequence which is an $\ell_1^{\omega^\xi}$ spreading model. By the previous paragraph, we may fix a block sequence $(z_j)_{j=1}^\infty \subset 2B_{Z_\xi(E_X)}$ such that $Qz_j=u_j$ for all $j\in\nn$.    Then using Proposition \ref{facts} and Corollary \ref{conn}, $(z_j)_{j=1}^\infty$ is weakly null in $Z_\xi(E_X)$ and has no subsequence which is an $\ell_1^{\omega^\xi}$, so $Z_\xi(E_X)$ is not $\omega^\xi$-Schur.

$(iii)$ Fix a seminormalized, weakly null sequence $(z_i)_{i=1}^\infty$ in $Z_\xi(E)$.  By passing to a subsequence,  perturbing, and scaling, we may assume $(z_i)_{i=1}^\infty$ is a normalized block sequence in $Z_\xi(E)$.  For each $i\in\nn$, let  $I_i=[\min \text{supp}(z_i), \max \text{supp}(z_i)]$.  Fix $n\in\nn$ and note that for any $F\in \mathcal{S}_{\xi_n}$ and any scalars $(a_i)_{i\in F}$, $(\min I_i)_{i\in F}\in \mathcal{S}_{\xi_n}$, whence $$\|\sum_{i\in F} a_iz_i\|_{Z_\xi(E)} \geqslant \Bigl(\sum_{j=1}^\infty \bigl\|\sum_{i\in F}a_iz_i\|_j^2\Bigr)^{1/2} \geqslant \vartheta_n \sum_{i\in F} \|I_i \sum_{j\in F}a_jz_j\|_{Z_\xi(E)} = \vartheta_n \sum_{i\in F}|a_i|.$$  Thus every seminormalized, weakly null sequence in $Z_\xi(E)$ has a subsequence which is an $\ell_1^{\xi_n}$ spreading model. From this it follows that $\textsf{v}(Z_\xi(E))>\xi_n$. Since $\sup_n \xi_n =\omega^\xi$, $\textsf{v}(Z_\xi(E))\geqslant \omega^\xi$.

$(iv)$ By $(ii)$, $\textsf{v}(Z_\xi(E_X))\leqslant \omega^\xi$, and by $(iii)$, $\textsf{v}(Z_\xi(E_X))\geqslant \omega^\xi$.

\end{proof}

In what follows, let $Z_0(c_0)=c_0$ and for any Banach space $E$, let  $Z_{\omega_1}(E)=\ell_1$.  

\begin{proposition} \begin{enumerate}[(i)]\item For any $\xi\leqslant \omega_1$, $\textsf{\emph{v}}(Z_\xi(c_0))=\omega^\xi$. \item  For any Banach space $Z$ and any closed subspace $Y$ of $Z$, $$\min \{\textsf{\emph{v}}(Y), \textsf{\emph{v}}(Z/Y)\}\leqslant \textsf{\emph{v}}(Z)\leqslant \textsf{\emph{v}}(Y).$$  \item For any Banach spaces $X,Y$, $$\textsf{\emph{v}}(X\oplus Y)=\min \{\textsf{\emph{v}}(X), \textsf{\emph{v}}(Y)\}.$$   \end{enumerate} 

\label{ds}
\end{proposition}

\begin{proof} $(i)$ It is clear that $\textsf{v}(c_0)=\textsf{v}(Z_0(c_0))=1=\omega^0$ and $\textsf{v}(Z_{\omega_1}(c_0))=\omega_1=\omega^{\omega_1}$.   For $0<\xi<\omega_1$, $\textsf{v}(Z_\xi(c_0))=\omega^\xi$ by Corollary \ref{beans}.

$(ii)$ Fix $\xi<\min\{\textsf{v}(Y), \textsf{v}(Z/Y)\}$ and a $\xi$-weakly null sequence $(z_i)_{i=1}^\infty$ in $Z$.  Then $(z_i+Y)_{i=1}^\infty$ is $\xi$-weakly null in $Z/Y$, and therefore norm null.  For all $i\in \nn$, we may fix $y_i\in Y$ such that $\|z_i-y_i\|\leqslant 2 \|z_i+Y\|_{Z/Y}$.  Then since $\lim_i \|z_i-y_i\|=0$,  $(y_i)_{i=1}^\infty\subset Y$ is $\xi$-weakly null, and therefore norm null.   From this it follows that $\lim_i z_i=0$, and $Z$ is $\xi$-Schur. 

Now if $j:Y\to X$ is the inclusion, then $I_X j=jI_Y\in \mathfrak{V}_\xi$. Since $\mathfrak{V}_\xi$ is an injective ideal, $\mathfrak{V}_\xi$, $\textsf{v}(I_Y) \geqslant \textsf{v}(I_X)$.

$(iii)$ This follows from $(ii)$.

\end{proof}

We now have the converse of Corollary \ref{block of block}.  

\begin{corollary} For any $\xi\leqslant \omega_1$, there exists a Banach space $X$ with $\textsf{\emph{v}}(X)=\xi$ if and only if there exists $\gamma\leqslant \omega_1$ such that $\xi=\omega^\gamma$.

\end{corollary}

\begin{theorem} \begin{enumerate}[(i)]\item for any $\xi, \zeta\leqslant \omega_1$, there is a Banach space $Z$ with closed subspace $Y$ such that $\textsf{\emph{v}}(Z)=\omega^\xi$ and $\textsf{\emph{v}}(Y)=\omega^\zeta$ if and only if $\xi\leqslant \zeta$.   \item For any $\xi, \zeta\leqslant \omega_1$, there is a Banach space $Z$ with a closed subspace $Y$ such that $\textsf{\emph{v}}(Z)=\omega^\xi$ and $\textsf{\emph{v}}(Z/Y)=\omega^\zeta$.  \item For any ordinals $\xi, \zeta, \eta\leqslant \omega_1$, there exists a Banach space $W$ with closed subspace $X$ such that $\textsf{\emph{v}}(W)=\omega^\xi$, $\textsf{\emph{v}}(X)=\omega^\zeta$, and $\textsf{\emph{v}}(W/X)=\omega^\eta$ if and only if either $\eta\leqslant \xi\leqslant \zeta$ or $\xi=\zeta$.  \end{enumerate}

\end{theorem}

\begin{proof}$(i)$ By Proposition \ref{ds}, if there exist Banach spaces $Y\leqslant Z$ such that $\textsf{v}(Z)=\omega^\xi$ and $\textsf{v}(Y)=\omega^\zeta$, then $\zeta\geqslant \xi$.  Conversely, for any $0\leqslant \xi\leqslant \zeta\leqslant \omega_1$, we may take $Y=Z_\zeta(c_0)$ and $Z=Z_\xi(c_0)\oplus Y$.

$(ii)$ Fix $\xi, \zeta\leqslant \omega_1$.  If $0\leqslant \xi\leqslant \zeta\leqslant \omega_1$, we let $Y=Z_\xi(c_0)$ and $Z=Z_\xi(c_0)\oplus Z_\zeta(c_0)$, so $Z/Y=Z_\zeta(c_0)$.   Now suppose that $0\leqslant \zeta<\xi\leqslant \omega_1$.  Then $Z_\zeta(c_0)$ is  a quotient of $Z_\xi(Z_\zeta(c_0))$.  We take $Y$ to be the kernel of any quotient map from $Z_\xi(Z_\zeta(c_0))$ onto $Z_\zeta(c_0)$.

$(iii)$ Suppose $W$ is a Banach space and $X$ is a closed subspace such that $\textsf{v}(W)=\omega^\xi$, $\textsf{v}(X)=\omega^\zeta$, and $\textsf{v}(W/X)=\omega^\eta$.  Then $$\min \{\textsf{v}(X), \textsf{v}(W/X)\} \leqslant \textsf{v}(W) \leqslant \textsf{v}(X).$$  If $\min \{\textsf{v}(X), \textsf{v}(W/X)\}=\textsf{v}(X)$, then $\xi=\zeta$. If $\min \{\textsf{v}(X), \textsf{v}(W/X)\}=\textsf{v}(W/X)$, then $\eta\leqslant \xi\leqslant \zeta$.  

Conversely, if $\eta\leqslant \xi\leqslant \zeta$, then $Z_\eta(c_0)$ is a quotient of $Z:=Z_\xi(Z_\eta(c_0))$.  We let $Y$ be the kernel of any quotient map from $Z$ onto $Z_\eta(c_0)$ and let $X=Z_\zeta(c_0)$ and $W=Z\oplus X$.   

If $\xi=\zeta$ and $0\leqslant \xi\leqslant \eta \leqslant \omega_1$, we may take $X=Z_\eta(c_0)$ and $W=Z_\xi(c_0)\oplus X$.

If $\xi=\zeta$ and $0\leqslant \eta<\xi$, we take $Z=Z_\xi(Z_\eta(c_0))$ and let $X$ be the kernel of any quotient map from $Z_\xi(Z_\eta(c_0))$ onto $Z_\eta(c_0)$.

\end{proof}

\section{A remark on descriptive set theoretic complexity}

Due to the brevity of this section, we refer the reader to \cite{Dodos} and \cite{BCFW} for a more detailed presentation of $\textbf{SB}$ and $\mathcal{L}$, respectively.   We recall that $C(2^\nn)$ denotes the space of all continuous, scalar-valued functions on the Cantor set $2^\nn$ and $\textbf{SB}$ denotes  the set of all closed subsets of $C(2^\nn)$  which are linear subspaces. The space $\textbf{SB}$ can be endowed with a Polish topology, and as such, we can consider the Effros-Borel structure of $\textbf{SB}$.     This is a coding of all separable Banach spaces, and it is of interest to study the descriptive set theoretic complexity of subsets of $\textbf{SB}$.  We also recall the existence of a sequence $d_n:\textbf{SB}\to C(2^\nn)$ of Borel functions such that for each $X\in \textbf{SB}$, $X=\overline{\{d_n(X): n\in\nn\}}$ \cite{KRN}.   As usual, we are not concerned with the particular Polish topology on $\textbf{SB}$, but only the Borel $\sigma$-algebra it generates, so we may select a specific topology on $\textbf{SB}$ which generates the Effros-Borel $\sigma$-algebra which is convenient for our purposes.   By standard techniques \cite[Theorem $13.1$, page 82]{Ke}, we may take a sequence $(U_n)_{n=1}^\infty$ of open subsets of $C(2^\nn)$ which form a neighborhood basis for the norm topology on $C(2^\nn)$ and fix a Polish topology on $\textbf{SB}$ which generates the Effros-Borel $\sigma$-algebra and contains each of the sets $d_m^{-1}(U_m)$, $m,n\in\nn$.    Therefore we can, and in the sequel do, assume a fixed topology on $\textbf{SB}$ such that for each $n\in\nn$, $d_n:\textbf{SB}\to C(2^\nn)$ is continuous.

We endow $\textbf{SB}\times \textbf{SB}\times C(2^\nn)^\nn$ with the product topology, which is also Polish.   We then let $\mathcal{L}$ denote the set of all triples $(X,Y, (y_n)_{n=1}^\infty)\in \textbf{SB}\times \textbf{SB}\times C(2^\nn)^\nn$ such that there exists $k\in\nn$ such that for any $n\in\nn$ and any scalars $(a_i)_{i=1}^n$, $$\|\sum_{i=1}^n a_i y_n\|\leqslant k \|\sum_{i=1}^n a_i d_i(X)\|.$$  Then $\mathcal{L}$ is a Borel (and, in fact, $F_\sigma$) subset of $\textbf{SB}\times \textbf{SB}\times C(2^\nn)^\nn$, and we may endow it with a Polish topology which generates the Effros-Borel $\sigma$-algebra (which is the product of the Borel $\sigma$-algebras) and which is stronger than the product topology coming from $\textbf{SB}\times \textbf{SB}\times C(2^\nn)^\nn$.   However, since we are not concerned with the particular topology on $\mathcal{L}$, only the corresponding  Borel $\sigma$-algebra, we will endow $\mathcal{L}$ with the product topology.

We recall that a subset $T$ of a Polish space $P$ is \begin{enumerate}[(i)]\item \emph{analytic} (or $\Sigma_1^1$) if there exist a Polish space $Y$, a Borel subset $B$ of $Y$, and a continuous function $f:Y\to P$ such that $f(B)=T$, \item \emph{coanalytic} (or $\Pi_1^1$) if its complement is $\Sigma_1^1$, \item $\Sigma_2^1$ if there exists a Polish space $Y$, a coanalytic subset $A$ of $Y$, and a continuous function $f:Y\to P$ such that $f(A)=Y$, \item $\Pi_2^1$ if its complement is $\Sigma_2^1$. \end{enumerate}

We say $\Gamma$ is a \emph{pointclass} if $\Gamma$ is a class of subsets of Polish spaces such that for every Polish space $P$, $\Gamma(P)\subset \Gamma$ is a subset of the power set of $P$.   For example, we let $\Sigma_1^1$ denote the class of all analytic subsets of Polish spaces and $\Sigma_1^1(P)$ denotes the set of analytic subsets of $P$.   In this case, if $P$ is a Polish space and $B\in \Gamma(P)$, we say $B$ is $\Gamma$-\emph{complete} provided that for any Polish space $Y$ and any $A\in \Gamma(Y)$, there exists a Borel function $f:Y\to P$ such that $f^{-1}(B)=A$.

We let $\textbf{Tr}$ denote the subset of $2^{\nn^{<\nn}}$ consisting of those subsets of $\nn^{<\nn}$ which are trees and $\textbf{WF}$ the subset of $\textbf{Tr}$ consisting of well-founded trees (that is, the trees $T\in \textbf{Tr}$ such that there do not exist $(n_i)_{i=1}^\infty\in \nn^\nn$ such that $(n_i)_{i=1}^k\in T$ for all $k\in \nn$).   

\begin{fact}\cite{Dodos} The set $\textbf{\emph{Tr}}$ is a Polish space and $\textbf{\emph{WF}}$ is $\Pi_1^1$-complete in $\textbf{\emph{Tr}}$.  
\label{f1}
\end{fact}

\begin{fact} If $\Gamma=\Pi_1^1$ or $\Pi_2^1$, $P,Y$ are Polish spaces,  $B\in \Gamma(P)$ is $\Gamma$-complete, $A\in \Gamma(Y)$, and there exists a Borel function $g:P\to Y$ such that $g^{-1}(A)=B$, then $A$ is $\Gamma$-complete. 
\label{f2}

\end{fact}

  We are now ready to prove the following.

\begin{theorem} For $0<\xi<\omega_1$, the class $\mathfrak{V}_\xi\cap \mathcal{L}$ is $\Pi_1^1$-complete in $\mathcal{L}$ and the class $\textsf{\emph{V}}_\xi\cap \textbf{\emph{SB}}$ is $\Pi_1^1$-complete in $\mathcal{L}$.

\end{theorem}

\begin{proof}   Let $A,B,C$, respectively, be the subsets of $P:=\mathcal{L}\times \nn^\nn\times \nn^\nn\times \nn$ consisting of those $$((X,Y, (y_n)_{n=1}^\infty), (r_i)_{i=1}^\infty, (m_i)_{i=1}^\infty, p)\in P$$ such that \begin{equation} m_1<m_2<\ldots \tag{$A$}\end{equation} \begin{equation} \text{for all\ }r\in \{r_1, r_2, \ldots\}, \|y_r\|\geqslant 1/p \tag{$B$}\end{equation}  \begin{equation} \text{\ for all\ }N\in[\{m_1, m_2, \ldots\}]\text{\ and\ }t\in\nn, \|\sum_{i=1}^\infty \textbf{S}^\xi_{N,t}(i)d_{r_i}(X)\|\leqslant 1/t. \tag{$C$}\end{equation} 

Standard arguments yield that $A,B$ are closed sets.    We will prove that $C$ is Borel, and in fact closed with respect to our particular choice of topology.    Suppose $((X,Y, (y_n)_{n=1}^\infty), (r_i)_{i=1}^\infty, (m_i)_{i=1}^\infty, p)$ lies in the complement of $C$.   This means there exist $N\in [\{m_1, m_2, \ldots\}]$ and $t\in \nn$ such that $$a:=\|\sum_{i=1}^\infty \textbf{S}^\xi_{N,t}(i)d_{r_i}(X)\|>1/t.$$  Let $E=\cup_{i=1}^t \text{supp}(\mathbb{S}^\xi_{N,i})$ and note that this is a finite initial segment of $N$ such that if $L$ is any infinite subset of $\nn$ with $E<L$, then $\mathbb{S}^\xi_{N,i}=\mathbb{S}^\xi_{E\cup L, i}$ for each $1\leqslant i\leqslant t$.     Let $U$ be a neighborhood of $X$ in $\textbf{SB}$ such that for each $i\in \nn$ with $r_i\in E$, $\|d_{r_i}(X)-d_{r_i}(Z)\|<\|\sum_{i=1}^\infty \mathbb{S}^\xi_{N,t}(i)d_{r_i}(X)\|-1/t$ for any $Z\in U$.  We may choose such a set $U$, since there are only finitely many such $i$ and the selectors $d_{r_i}$ are continuous. Fix any $k\in\nn$ such that $\text{supp}(\mathbb{S}^\xi_{N, t})\subset \{m_1, \ldots, m_k\}$ and $E\subset \{1, \ldots, k\}$.   Let $$V=\{L\in [\nn]: L\cap \{1, \ldots, m_k\}= \{m_1, \ldots, m_k\}\}$$ and $$W=\{L\in[\nn]: L\cap \{1, \ldots, r_k\}= \{r_1, \ldots, r_k\}\}.$$  These are open neighborhoods of $M$ and $R$, respectively.   Finally, let $$H=\{((Z_1, Z_2, (z_n)_{n=1}^\infty), (s_i)_{i=1}^\infty, (j_i)_{i=1}^\infty, q)\in \mathcal{L}\times \nn^\nn\times \nn^\nn\times \nn: Z_1\in U, (s_i)_{i=1}^\infty\in W, (j_i)_{i=1}^\infty \in V\}.$$  Then $H$ is an open neighborhood of $((X,Y,(y_n)_{n=1}^\infty),(r_i)_{i=1}^\infty,  (m_i)_{i=1}^\infty,p)$ and $H\cap C=\varnothing$. This yields the openness of the complement of $C$, and the closedness of $C$.   Let us see that $H\cap C=\varnothing$.  If $$((Z_1, Z_2, (z_n)_{n=1}^\infty), (s_i)_{i=1}^\infty, (j_i)_{i=1}^\infty, p)\in H,$$ then $E$ is an initial segment of some infinite subset $L$ of $\{s_1, s_2, \ldots\}$.    Furthermore, $\mathbb{S}^\xi_{L,i}=\mathbb{S}^\xi_{N,i}$ for all $1\leqslant i\leqslant t$, $r_i=j_i$ for all $i\in \cup_{m=1}^t \text{supp}(\mathbb{S}^\xi_{N,m})=\cup_{m=1}^t \text{supp}(\mathbb{S}^\xi_{L,m})$, and $\|d_{r_i}(X)-d_{r_i}(Z_1)\|<a-1/t$ for all such $i$.  From this it follows that \begin{align*} \|\sum_{i=1}^\infty \mathbb{S}^\xi_{L,t}(i)d_{j_i}(Z_1)\| & \geqslant \|\sum_{i=1}^\infty \mathbb{S}^\xi_{N,t}(i)d_{j_i}(Z_1)\|- \sum_{i\in \text{supp}(\mathbb{S}^\xi_{N,t})} \mathbb{S}^\xi_{N,t}(i) \|d_{r_i}(Z_1)-d_{r_i}(X)\| \\ & > 1/t. \end{align*}

Now we let $\pi:\mathcal{L}\times\nn^\nn\times\nn^\nn\times \nn\to \mathcal{L}$ be the projection and note that by Proposition \ref{and another thing}, $\pi(A\cap B\cap C)$ is the complement of the $\xi$-completely continuous operators.  This yields that $\mathfrak{V}_\xi\cap \mathcal{L}$ is $\Pi_1^1$.   Now it is easy to see that the map $\Phi:\textbf{SB}\to \mathcal{L}$ given by $\Phi(X)=(X,X,(d_n(X))_{n=1}^\infty)$ is Borel and $\Phi^{-1}(\mathfrak{V}_\xi\cap \mathcal{L})= \textsf{V}_\xi\cap \textbf{SB}$, yielding that $\textsf{V}_\xi\cap \textbf{SB}$ is $\Pi_1^1$.

The usual examples (which are modifications of examples appearing in \cite{Boss}) serve to show that $\textsf{V}_\xi\cap \textbf{SB}$ is $\Pi_1^1$-complete.   We define the map $T\mapsto J^T$, where $J$ is any fixed subspace of $\textbf{SB}$ isometrically isomorphic to the completion of $c_{00}([\nn]^{<\nn}\setminus \{\varnothing\})$ endowed with the norm $$\|\sum_{t\in [\nn]^{<\nn}\setminus \{\varnothing\}} a_t e_t\|_J = \sup\Bigl\{\sum_{i=1}^n \max_{t\in \mathfrak{s}_i} |a_t|: n\in\nn, \mathfrak{s}_1, \ldots, \mathfrak{s}_n\text{\ are incomparable segments}\Bigr\},$$ and $J^T$ is the closed span in $J$ of $[e_t: t\in T]$.    Then if $T$ is well-founded, $J^T$ has the Schur property, and if $T$ is ill-founded, $J^T$ contains a copy of $c_0$ and is not $1$-Schur.   Furthermore, the map $\tau(T)=J^T$ satisfies $\tau^{-1}(\textsf{V}_\xi)=\tau^{-1}(\textsf{V})=\textbf{WF}$.  Now we note that since $\Phi:\textbf{SB}\to \mathcal{L}$ as given in the previous paragraph is Borel and $\Phi^{-1}(\mathfrak{V}_\xi\cap \mathcal{L})=\textsf{V}_\xi\cap \textbf{SB}$, Fact \ref{f2} yields that $\mathfrak{V}_\xi \cap \mathcal{L}$ is $\Pi_1^1$-complete.

\end{proof}

\begin{proposition} The class $\mathfrak{V}\cap \mathcal{L}$ is $\Pi_2^1$.   
\label{new prop}
\end{proposition}

\begin{proof} Let $A,B$, respectively, denote the subsets of $P:=\mathcal{L}\times \nn^\nn\times \nn $ consisting of those $$((X,Y, (y_i)_{i=1}^\infty), (n_i)_{i=1}^\infty, p)\in P$$ such that \begin{equation} n_1<n_1<\ldots,\tag{A} \end{equation} \begin{equation} \text{for all\ }i\in\nn, \|y_{n_i}\|\geqslant 1/p\text{\ and\ }\|d_{n_i}(X)\|\leqslant 1. \tag{B} \end{equation}

Let $C$ denote the set of all $((X,Y, (y_i)_{i=1}^\infty), (n_i)_{i=1}^\infty, k,p) \in \mathcal{L}\times \nn^\nn\times \nn\times \nn$ such that for all $l\in \nn$ and all positive scalars $(a_i)_{i=1}^l$ summing to $1$, $\|\sum_{i=1}^l a_i d_{n_i}(X)\|\geqslant 1/k$.   It is evident that $C$ is a closed set.   Now let $F:\mathcal{L}\times \nn^\nn\times \nn\times\nn\to \mathcal{L}\times \nn^\nn\times \nn$ denote the projection $$F((X,Y, (y_i)_{i=1}^\infty), (n_i)_{i=1}^\infty, k,p)=((X,Y, (y_i)_{i=1}^\infty), (n_i)_{i=1}^\infty, p)$$ and let $D=F(C)$.   Then $D$ is analytic.  Furthermore, $((X,Y, (y_i)_{i=1}^\infty), (n_i)_{i=1}^\infty, p)\in A\cap B\cap (P\setminus D)$ if and only if $n_1<n_2<\ldots$, $(d_{n_i}(X))_{i=1}^\infty$ is weakly null, and $\|y_{n_i}\|\geqslant 1/p$ for all $i\in\nn$.    Now if $\pi:P\to \mathcal{L}$ is the projection onto $\mathcal{L}$, then $\pi(A\cap B\cap (P\setminus D))= \mathcal{L}\setminus \mathfrak{V}$. Since $A,B$ are Borel and $D$ is analytic, $A\cap B\cap (P\setminus D)\in \Pi_1^1$, and $\mathcal{L}\setminus \mathfrak{V}=\pi(A\cap B\cap (P\setminus D))\in \Sigma_2^1$.

\end{proof}

\begin{rem}\upshape Kurka \cite{Kurka} showed that $\textsf{V}\cap \textbf{SB}$ is $\Pi_2^1$-complete.   Since $\textsf{V}\cap \textbf{SB}=\Phi^{-1}(\mathfrak{V}\cap \mathcal{L})$ and $\Phi$ is Borel, $\mathfrak{V}\cap \mathcal{L}$ is $\Pi_2^1$-complete by Fact \ref{f2}.    This yields that for each $0<\xi<\omega_1$, the classes $\mathfrak{V}_\xi\cap \mathcal{L}$ have strictly lower complexity than $\mathfrak{V}\cap \mathcal{L}$.   The same holds for the classes $\textsf{V}_\xi\cap \textbf{SB}$, $0<\xi<\omega_1$. 

\end{rem}


\begin{thebibliography}{HD}

\normalsize
\baselineskip=17pt

\bibitem{ADST} G. Androulakis, P. Dodos, G. Sirotkin, V.G. Troitsky, \emph{Classes of strictly singular operators
and their products}, Israel J. Math. \textbf{169} (2009), 221-250.

\bibitem{AB} S. Argyros, K. Beanland, \emph{On spaces admitting no $\ell_p$ or $c_0$ spreading model},  Positivity \textbf{17}(2) (2013),  265-282.

\bibitem{AG} S. A. Argyros,  I. Gasparis,  \emph{Unconditional structures of weakly null sequences},  Trans. Amer. Math. Soc.
\textbf{353}(5) (2001), 2019-2058.


\bibitem{AMT}  S. A. Argyros, S. Mercourakis,  A. Tsarpalias, \emph{ Convex unconditionality and summability
of weakly null sequences}, Israel J. Math. \textbf{107} (1998), 157-193.


\bibitem{BC} K. Beanland, R.M. Causey, \emph{Quantitative factorization of weakly compact, Rosenthal, and  $\xi$-Banach-Saks operators}, to appear in Mathematica Scandinavica.  

\bibitem{BC2} K. Beanland, R.M. Causey, \emph{Genericity and Universality for Operator Ideals}, submitted. 

\bibitem{BCFW} K. Beanland, R.M. Causey, D. Freeman, B. Wallis, \emph{Classes of operators determined by ordinal indices}, J. Funct. Anal. \textbf{271} (6) (2016), 1691-1746.


\bibitem{BCM} K. Beanland, R.M. Causey, P. Motakis, \emph{Arbitrarily distortable Banach spaces of higher order}, Israel J. of Math. \textbf{214} (2016), 553-581.


\bibitem{BF} K. Beanland, D. Freeman, \emph{Uniformly factoring weakly compact operators}, J. Functional Anal., \textbf{266} (2014), no. 5, 2921-2943.

\bibitem{Boss} B. Bossard, \emph{Th\'{e}orie descriptive des ensembles en g\'{e}eom\'{e}trie des espaces de Banach}. PhD thesis, 1994.

\bibitem{Brooker} P.A.H. Brooker, \emph{Asplund operators and the Szlenk index}, J. Operator Theory \textbf{68} (2012), 405-442.




\bibitem{Concerning} R. M. Causey, \emph{Concerning the Szlenk index}, Studia Math. \textbf{236} (2017), 201-244.  


\bibitem{C1} R. M. Causey, \emph{Proximity to $\ell_p$ and $c_0$ in Banach spaces}, J. Funct. Anal \textbf{269} (12) (2015), 3952-4005. 


\bibitem{Cweakly} R.M. Causey, \emph{An ordinal index characterizing weak compactness of operators}, to appear in Fundamenta Mathematicae. 

\bibitem{DS}  J. Diestel,  C. J. Seifert, \emph{The Banach-Saks ideal, I. Operators acting on $C(\Omega)$}, Comment.
Math. 1 (1978), 109-118.

\bibitem{Dodos} P. Dodos,  \emph{Banach spaces and descriptive set theory: selected topics}, volume 1993 of Lecture Notes in Mathematics.
Springer-Verlag, Berlin, 2010.


\bibitem{Ellentuck} E. Ellentuck, \emph{A new proof that analytic sets are Ramsey}, J. Symbolic Logic \textbf{39} (1974), 163-165.


\bibitem{GP} F. Galvin and K. Prikry, \emph{Borel sets and Ramsey’s theorem}, J. Symbolic Logic \textbf{38} (1973), 193-198.


\bibitem{Ga} I. Gasparis, \emph{A dichotomy theorem for subsets of the power subsets of
the power set of the natural numbers}, Proc. Am. Math. Soc. \textbf{129}, (2001), 759-764. 


\bibitem{Gr} A. Grothendieck, \emph{Sur les applications lin\'{e}aires faiblement compactes d'espaces du type $C(K)$}, Canad. J. Math. \textbf{5} (1953), 129-173. 

\bibitem{JLO} W.B. Johnson, R. Lillemets, E. Oja, \emph{Representing completely continuous operators through weakly $\infty$-compact operators}, submitted. 

\bibitem{Ke} A. S. Kechris. \emph{Classical descriptive set theory}, volume 156 of Graduate Texts in Mathematics. Springer-Verlag,
New York, 1995.


\bibitem{Kurka} O. Kurka, \emph{Tsirelson-like spaces and complexity of classes of Banach spaces}, DOI arXiv:1612.09334.  

\bibitem{NW} C. St. J. A. Nash-Williams, \emph{On well quasi-ordering transfinite sequences}, Proc. Camb. Phil. Soc. \textbf{61} (1965), 33-39.

\bibitem{OS} E. Odell, Th. Schlumprecht,  \emph{On the richness of the set of p’s in Krivine’s theorem},
Geometric aspects of functional analysis (Israel, 1992-1994), volume 77 of Oper. Theory
Adv. Appl., 177-198. Birkh\:{a}auser, Basel, (1995).


\bibitem{Pelc} A. Pe\l czy\'{n}ski, \emph{Banach spaces on which every unconditionally converging operator is weakly compact}, Bull. Acad. Pol. Sci. S\'{e}r. Sci., Math., Astr. et Phys. \textbf{12} (1962), 641-648. 

\bibitem{PR} P. Pudl\'{a}k, V. R\"{o}dl, \emph{ Partition theorems for systems of finite subsets of integers}, Discrete Math. \textbf{39}(1) (1982), 67-73. 

\bibitem{KRN} K. Kuratowski, C. Ryll-Nardzewski, \emph{A general theorem on selectors}, Bull. Acad. Pol. Sci. S\'{e}r. Sci., Math., Astr. et Phys. \textbf{13} (1965), 397-403. 

\bibitem{Schlumprecht} Th. Schlumprecht, \emph{On Zippin’s Embedding Theorem of Banach spaces into Banach spaces with bases}, Adv. Math., \textbf{274} (2015), 833-880.

\bibitem{Silver}  J. Silver, \emph{Every analytic set is Ramsey}, J. Symbolic Logic \textbf{35} (1970), 60-64.








\end{thebibliography}
\end{document}